\documentclass[11pt]{amsart}

\usepackage{amssymb, amsfonts, amsmath, amsthm}
\usepackage{wrapfig}
\usepackage{graphicx}
\usepackage{enumerate}

\usepackage{bbm}
\newcommand{\1}{\mathbbm{1}} 

\newcommand{\R}{\mathbb{R}}
\newcommand{\E}{\mathbb{E}}
\newcommand{\CC}{\mathcal C}
\newcommand{\PP}{\mathcal P}
\newcommand{\TT}{\mathcal{T}} 
\newcommand{\eps}{\varepsilon}
\newcommand{\Cl}{\mathcal{C}\ell}
\newcommand{\Var}{\mathbb{V}\mathrm{ar}}
\newcommand{\bP}{\mathbb{P}}
\renewcommand{\P}{\bP} 
\def\pcv{\stackrel{\scriptscriptstyle \P}{\longrightarrow}}        

\newcommand{\bal}[1]{\begin{align*}#1\end{align*}}

\renewcommand{\emph}[1]{\textbf{\textit{#1}}}
\newcommand{\ovln}[1]{\overline{#1}}

\theoremstyle{plain}
\newtheorem{theo}{Theorem}[section]
\newtheorem{coro}[theo]{Corollary}
\newtheorem{propo}[theo]{Proposition}
\newtheorem{lemma}[theo]{Lemma}

\theoremstyle{definition}
\newtheorem{defi}[theo]{Definition}

\begin{document}

\begin{center}
{\Large Cycles in random meander systems}

\bigskip
Vladislav Kargin\footnote{{email:
vkargin@binghamton.edu; current address: 4400 Vestal Pkwy East, Department of Mathematics, Binghamton University, Binghamton, 13902-6000, USA}} 

\bigskip
Abstract: 
\end{center}

\begin{quote}
{\small
A meander system is a union of two arc systems that represent non-crossing pairings of the set $[2n] = \{1, \ldots, 2n\}$ in the upper and lower half-plane. In this paper, we consider random meander systems. We show that for a class of random meander systems, -- for simply-generated meander systems, -- the number of cycles in a system of size $n$ grows linearly with $n$ and that the length of the largest cycle in a uniformly random meander system grows at least as $c \log n$ with $c > 0$. We also present numerical evidence suggesting that in a simply-generated meander system of size $n$, (i) the number of cycles of length $k \ll n$ is $\sim n k^{-\beta}$, where $\beta \approx 2$, and (ii) the length of the largest cycle is $\sim n^\alpha$, where $\alpha$ is close to $4/5$.  We compare these results with the growth rates in other families of meander systems, which we call rainbow meanders and comb-like meanders, and which show significantly different behavior. 
}
\end{quote}

%

\section{Setup: Random meander systems}
Let $P = \{(a_1, b_1),  \ldots, (a_n, b_n)\}$ be a pairing on the set  $[2n] = \{1, \ldots, 2n\}$. Pairing $P$  is called \emph{non-crossing} (``NC'') if there are no $\alpha < \beta < \gamma < \delta$ such that $\alpha$ is paired with $\gamma$ and $\beta$ with $\delta$. A non-crossing pairing can be realized by a family of $n$ non-intersecting plane arcs which connect $2n$ points $(1, 0), (2,0) \ldots, (2n, 0) \in \R^2$ and which are all in the upper (or, alternatively, all in the lower) half-plane.  

Let two non-crossing pairings $P_1$ and $P_2$ of $[2n]$  be realized by arcs in the upper and lower half-plane, respectively. We call the resulting system of curves a \emph{meander system}   $M = (P_1, P_2)$. (Sometimes we shorten ``meander system'' to \emph{m.s.}) The pairings $P_1$ and $P_2$ are called the upper and lower pairings of $M$, respectively. 

A \emph{cycle} of a meander system is one of its connected components. 

A meander system $M$ with only one cycle is called a (closed) \emph{meander}. Meanders have first appeared in Poincare's research on dynamical systems on surfaces, and the study of meanders was made popular by V. I. Arnold (\cite{arnold88}, see also a review in \cite{lando_zvonkin92}). The problem of meander enumeration is still not solved although there are very precise predictions based on numerics and on the conformal field theory in physics (\cite{dgg2000}). Recently, some research appeared on meander enumeration in situations, where the combinatorial complexity of the upper and lower pairings is constrained (\cite{dgzz2019}). 

In this paper, we consider meander systems that can have more than one cycle. In particular, we consider a \emph{random} meander system $M_n$ and study basic statistical properties of its cycles as $n \to \infty$. 

 We are mostly interested in the situation when a meander system $M_n$ is chosen with equal probability from all meander systems on $[2n]$. This is equivalent to choosing upper and lower pairings independently and with the uniform distribution. 

We consider meander systems with other distributions as well, in particular simply generated, random comb-like, and random rainbow m.s. They will be defined later and will be used for comparison with the results for uniformly distributed m.s. When we say ``a random meander system'' without qualification we mean a uniformly distributed m.s. 
 
 \begin{wrapfigure}{l}{.4 \textwidth}
              \includegraphics[width= \linewidth]{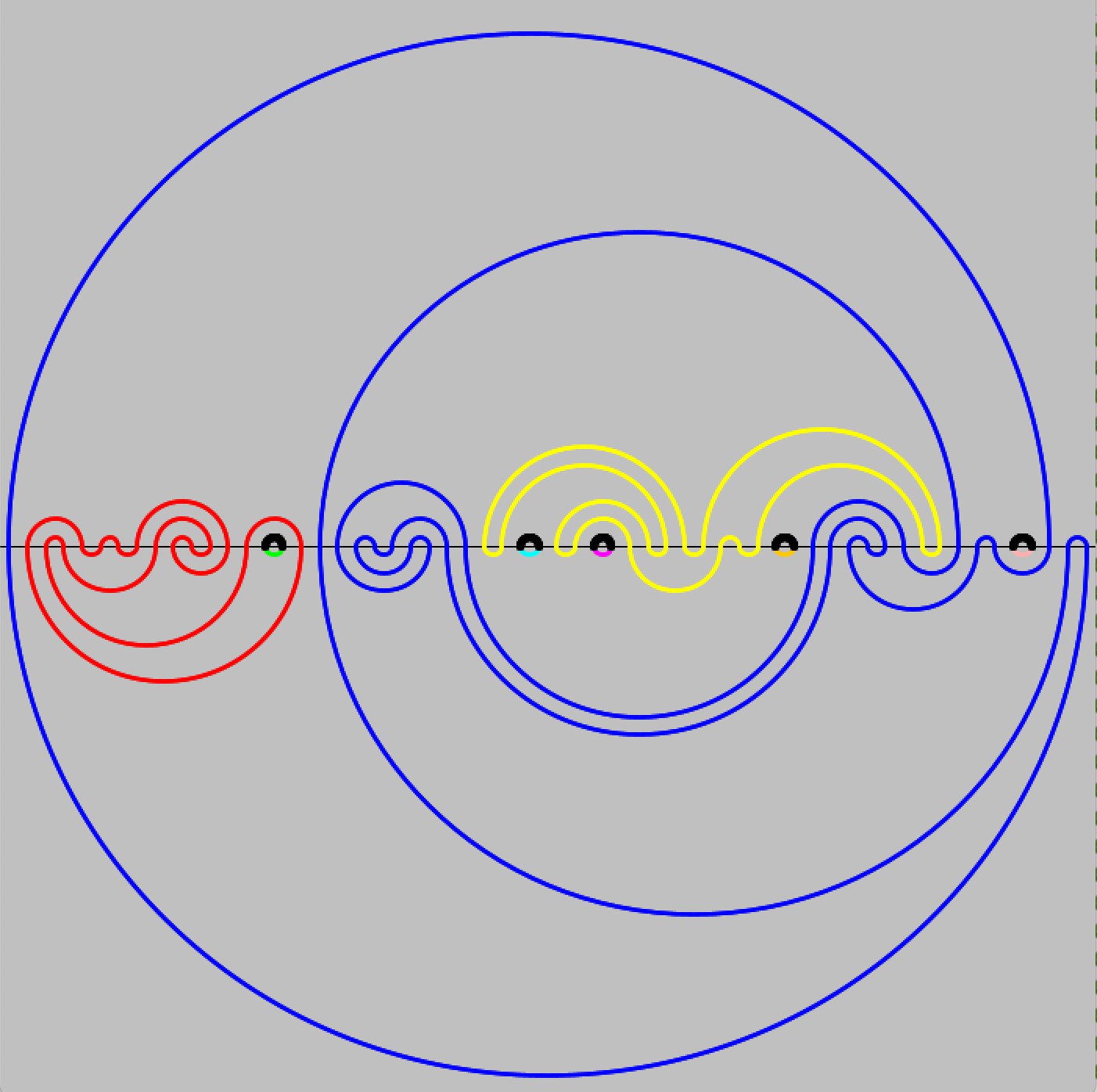}
              \caption{A meander system on $2 \times 30$ points with $8$ cycles.}
              \label{figMeander}
\end{wrapfigure}

 \textbf{Notation}
 The notation $g_n = O(f_n)$ means that $\limsup_{n \to \infty} |g_n/f_n| < \infty$; and $g_n = o(f_n)$ means that $\lim_{n \to \infty} |g_n/f_n| = 0$. We write $f_n = \Omega(g_n)$ as meaning the same thing as $g_n = O(f_n)$, and  we write $f_n \sim g_n$ to denote that $ f_n = g_n (1 + o(1))$. The notation $f_n \ll g_n$ means that $f_n = g_n o(1)$, and $f_n \approx g_n$ means $\log f_n \sim \log g_n$. 

 The remainder of the paper is organized as follows. Section \ref{sectionNumberCycles} discusses the number of cycles in random meander systems. Section \ref{sectionLargestCycle} is about the length of the largest cycle in a random m.s. Section \ref{sectionRemarks} provides some additional numerical evidence and concludes with remarks. And Appendix provides proofs which were relegated from the main body of the paper. 

\section{Number of cycles}
\label{sectionNumberCycles}
\subsection{Random meander systems with the uniform distribution}
 Numerics suggest that when the number of points $2n$ grows, the number of cycles $c(M_n)$ in a random meander system $M_n$ converges in distribution to a Gaussian limit, 
\bal{
\frac{c(M_n) - a n}{b \sqrt{n}} \to \mathcal{N}(0, 1),  
}
for some positive constants $a$ and $b$ (with $a \approx 0.23$ and $b \approx 0.42$.)   This limiting behavior is not easy to prove, and we are only going to show that the expectation of $c(M_n)$ grows at least linearly with $n$. 

A related study \cite{fukuda_nechita2019} considers the enumeration of meander systems that have the number of cycles $c(M_n) = n - r$, where $r \geq 0$ is a fixed integer. Their formulas imply that the probability of a random m.s. with $c(M_n) = n - r$ is exponentially small $\sim c_r n^{3/2 + r} 4^{-n}$.

In order to prove that the number of cycles grows at least linearly, we are going to count the number of special cycles which we call ``ringlets''. A \emph{ringlet} $\mathfrak{O}_x$, $1 \leq x \leq 2n - 1$, is a connected component that consists of arcs that connect $x$ and $(x + 1)$ both in upper and lower pairings.  For example, there are $5$ ringlets in Figure \ref{figMeander}.

\begin{theo}
\label{theoNumberRinglets}
Let $M_n$ be a (uniformly distributed) random m.s. on $2n$ points, and let $c_o(M_n)$ denote the total number of ringlets in $M_n$.  Then, $\E c_o(M_n) > (n + 1)/8$. Moreover, for $n \to \infty$, 
\bal{
\E c_o(M_n) \sim \frac{n}{8}.
}
\end{theo}

Since the number of cycles greater than the number of ringlets, Theorem \ref{theoNumberRinglets} immediately implies the desired result about the number of cycles.  
\begin{coro}
\label{coroNumberCycles}
Suppose $M_n$ is a random meander system on $2n$ points, and let $c(M_n)$ denote the number of cycles of $M_n$. Then, 
\bal{
\E \, c(M_n) > \frac{n + 1}{8} 
}
\end{coro}

\begin{proof}[Proof of Theorem \ref{theoNumberRinglets}]
Let $I_x^u$ be the indicator of the event that the upper pairing contains the pair $(x, x+1)$. Define indicator $I^l_x$ similarly for the lower pairing. Then, $I^u_x I^l_x  $ is the indicator for the event that the meander system contains the ringlet $\mathfrak{O}_x$. We write, 
\bal{
\E c_o(M_n)  =\E \sum_{x = 1}^{2n - 1} I^u_x I^l_x = \sum_{x = 1}^{2n - 1} \E(I^u_x) \E(I^l_x)
= \sum_{x = 1}^{2n - 1} p_n(x)^2, 
}
where $p_n(x)$ is the probability that arc $(x, x+1)$ is in the random NC pairing of $[2n]$.

The number of all non-crossing pairing is the Catalan number $C_n = \binom{2n}{n}/(n+1)$, and the number of non-crossing pairings of the set $[2n] \backslash \{x, x+1\}$ is the Catalan number $C_{n - 1}$. (The arc from $x$ to $x+1$ does not create any additional obstructions to non-crossing pairings of this set.) Hence,  
\bal{
p_n(x) = \frac{C_{n - 1}}{C_{n}} = \frac{n + 1}{2(2n - 1)}.
}
Therefore, 
\bal{
\E c_o(M_n) = \sum_{x = 1}^{2n - 1} \Big(  \frac{n + 1}{2(2n - 1)}\Big)^2
= \frac{1}{8}\frac{(n + 1)^2}{n-\frac{1}{2}} > \frac{n + 1}{8}.
}
It is also clear from this formula that $\E c_o(M_n) \sim n/8$.  
\end{proof}

The proof gives an asymptotic expression for the expected number of ringlets,  $\E c_o(M_n) \sim n/8$.
One can further estimate the expected number of cycles for other shapes and show that these expected numbers grow linearly with $n$. The calculations increase in difficulty when the shapes become more complex.

For example, consider a subclass of all cycles of half-length 2, which we call \emph{staples}, where \emph{upper staple} $[x, y]$, with $x < y  - 1$, is a cycle that consists of two pairs, $(x, x+1)$ and $(y, y+1)$, in the lower pairing, and two pairs, $(x, y+1)$ and $(x + 1, y)$, in the upper pairing. A \emph{lower staple} is defined similarly by exchanging the roles of lower and upper pairings. 

The following proposition is proved in the appendix. 
\begin{propo}
\label{propoStaples}
Let $M_n$ be a random m.s. on $2n$ points, and let $c_{st}(M_n)$ denote the total number of lower and upper staples in $M_n$.  Then, for $n \to \infty$, 
\bal{
\E c_{st}(M_n) \sim \frac{n}{32}.
}
\end{propo}

For another example, all cycles of half-length 1 are rings: A \emph{ring} of radius $r$ is a cycle that consists of a pair $(x, y)$ present both in the upper and lower pairings, with $y - x = 2r + 1$. (A ringlet is a ring of radius $r = 0$.) 

The following result gives an estimate on the expected number of rings. It is proved in the appendix.
\begin{propo}
\label{propoRings}
Let $M_n$ be a random m.s. on $2n$ points and let $c_O(M_n)$ denote the number of rings in $M_n$.  Then, for $n \to \infty$, 
\bal{
\E c_O(M_n) \sim n \Big(\frac{2}{\pi} - \frac{1}{2}\Big).
}
\end{propo}  
Numerically, $\E c_O(M_n) \sim (n/8) \times 1.092958$, so the rings of radius $\geq 1$ (``non-ringlet'' rings) are responsible for only $\approx 9.3\%$ of all rings in a random meander system.


Before proceeding further, let us introduce some additional definitions. 

The \emph{support} $(x_1 < x_2 < \ldots < x_{2k})$ of a cycle $\CC$ in a m.s. $M$ is the intersection of the cycle with the line $y = 0$. 

The \emph{half-length} of a cycle $\CC$ with support   $(x_1, \ldots, x_{2k})$ is the number $k$. 

Two cycles $\CC_1$ and $\CC_2$ are \emph{topologically equivalent} if there is an homeomorphism of the plane that maps $\CC_1$ to $\CC_2$ and maps the upper and lower half-planes to themselves. In this case we say that they have the same \emph{shape}.

For example, all rings are topologically equivalent. All cycles of half-length 2 are topologically equivalent to an upper or a lower staple. An upper and a lower staples have different shapes because  the reflection across the line $y = 0$ (or a rotation by the angle $\pi$ around a point on this line) does not preserve the upper and lower half-planes.

The shape of a cycle $\CC$ can be defined in purely combinatorial terms. If $(x_1 < x_2 < \ldots < x_{2k})$ is the support of $\CC$, then an oriented path along $\CC$ in the clockwise direction starting from $x_1$ gives a sequence in which the points of the support are visited. This sequence determines a permutation $\pi \in S_{2k}$: $(x_1 \to x_{\pi(1)} \to x_{\pi \circ \pi (1)} \to \ldots \to x_{\pi^{2k - 1}(1)} \to x_1)$.

This permutation is invariant under the homeomorphisms of the plane that preserve the upper and lower half-planes, and it determines the shape of the cycle. 

We will label the shape of a cycle by this permutation. 

For example, the shape of any ring is $\pi = (1 2)$, the shape of an upper staple is $(1 4 3 2)$, the shape of a lower staple is $(1 2 3 4)$. Note that by the connectedness of the cycle $\CC$, the permutation is a ``cycle'' permutation in $S_{2k}$, (where $k$ is the half-length of $\CC$). However, the non-crossing condition rules out some of the cycles in $S_{2k}$. For example, $(1324)$ is not possible. In fact, it is clear that the shape permutations can be put in a bijection with ``proper'' meanders on the set $[2k]$, that is, with meander systems that consist of only one connected component. 

   Let  $f_n \gtrsim g_n$ mean that there is a sequence $\eps_n \to 0$, such that $f_n \geq g_n (1  + \eps_n)$  for all $n$. Similarly  $f_n \lesssim g_n$ means that there is $\eps_n \to 0$ such that $f_n \leq g_n(1 + \eps_n)$ for all $n$. 

\begin{theo}
\label{theoLowerBound}
Let $c_{k,\pi}(M_n)$ be the number of cycles of half-length $k$ with  shape $\pi \in S_{2k}$ which occur in a (uniformly distributed) random m.s. $M_n$ on $2n$ points. 
Then, for every $k\geq 1$, 
\bal{
\E c_{k,\pi}(M_n) \gtrsim 2 n \times 16^{-k},
}
\end{theo}

In the proof we use the concept of a cluster cycle. 

\begin{defi}
\label{defiClusterCycle} A \emph{cluster} cycle of half-length $k$ is a cycle of half-length $k$ with support $(x + 1, x +2, \ldots, x+2k)$ for some $x \in \{0, \ldots, 2n - 2k\}$.
\end{defi}
  In other words, a cluster cycle has no gaps in its support. For example, a ringlet is a cluster cycle, while all other rings with radius $k \geq 1$ are not cluster cycles.

\begin{proof}[Proof of Theorem \ref{theoLowerBound}]
Instead of counting all cycles of half-length $k$ that have shape $\pi$, we will compute the expected number of all cluster cycles with these half-length and shape. Call this number $c_{k,\pi}^{(cl)}(M_n)$. Let $\CC_{x,\pi}$ be the cluster cycle with shape $\pi$ supported on $(x+1, x + 2, \ldots, x + 2k)$ where $0 \leq x \leq 2n - 2k$. Let $I_{x,\pi}$ be the indicator of the event that a random m.s. contains cycle $\CC_{x,\pi}$. Then, we have $\E c_{k,\pi}(M_n) \geq \E c_{k,\pi}^{(cl)}(M_n)$, and 
\begin{eqnarray}
\label{asymptotic}
 \E c_{k,\pi}^{(cl)}(M_n)=  \sum_{x = 0}^{2n - 2k} \E(I_{x,\pi}) 
\\
\notag
= \sum_{x = 0}^{2n - 2k} \Big(\frac{C_{n - k}}{C_n}\Big)^2
= (2n - 2k + 1) \Big(\frac{C_{n - k}}{C_n}\Big)^2 \sim 2n 16^{-k}.
\end{eqnarray}
where the second equality follows because $\E(I_{x,\pi})$ is the probability that the random m.s. contains $\CC_{x,\pi}$ and this probability is $C_{n - k}^2/C_n^2$. Indeed,  there are exactly $C_{n - k}^2$ possible ways to arrange the upper and lower non-crossing pairings of the residual set $\{1, \ldots, x,  x+2k +1, \ldots, 2 n\}$ and every of these pairings remains non-crossing if we add the pairings from the cycle $\CC_{x,\pi}$. (Here we use the fact that $\CC_{x,\pi}$ is a cluster cycle.) Therefore, the probability is  $C_{n - k}^2/C_n^2$.

The asymptotic (\ref{asymptotic}) for $c_{k,\pi}^{(cl)}(M_n)$  implies the statement of the theorem. 
\end{proof}

We have seen that the count of rings is dominated by ringlets. So, in general we can ask the question if the count of cycles is dominated by cluster cycles. We can use the asymptotic in the proof of Theorem \ref{theoLowerBound} to obtain an upper bound on the expected number of cluster cycles. 

The number of possible shapes of a cycle with the half-length $2k$ equals the number of ``proper'' meanders, that is, the meander systems with only one cycle. Let $R_k$ denote the number of proper meanders on the set $[2k]$. (We use the notation $R_k$ instead of more traditional $M_k$ to avoid the confusion with our notation for meander systems.) For these numbers, it is known that they are super-additive: $R_k R_l \leq R_{k + l}$. By Fekete's lemma, this implies the existence of the limit 
\bal{
R = \lim_{k \to \infty} (R_k)^{1/k} = \sup_{k > 1} (R_k)^{1/k}.
}
In particular, by the second equality, 
\begin{equation}
\label{upperboundRk}
R_k \leq R^k, \text{ for all } k\geq 1.
\end{equation}
 In \cite{lando_zvonkin92}, Lando and Zvonkin calculated the growth rate for a related class of irreducible meander systems and conjectured that for proper meanders, $R = 12.26 \ldots$.   In \cite{dgg2000}, it is conjectured that   $R_k \sim c k^{-\alpha} R^k$ with an explicit formula for $\alpha > 0$. Theoretical bounds for $R$ have been given  in \cite{lando_zvonkin92} and \cite{albert_paterson2005}.  In particular, the lower and upper bounds  $11.380 \leq R \leq 12.901\ldots$ are given in \cite{albert_paterson2005}.
 
 With respect to numerical evidence, Table 1 in \cite{lando_zvonkin93} lists meander numbers up to $R_{14} = 61606881612$ (with a reference to calculations by Reeds and Shepp). These numbers are confirmed in \cite{franz_earnshaw2002}. In \cite{jensen2000} and \cite{jensen_guttmann2000}, the calculations have been extended to $R_{24}$ by using a different method. Table 1 in \cite{jensen_guttmann2000} gives $R_{24} = 794337831754564188184$.

\begin{figure}[htbp]
\centering
              \includegraphics[width= 0.8 \textwidth]{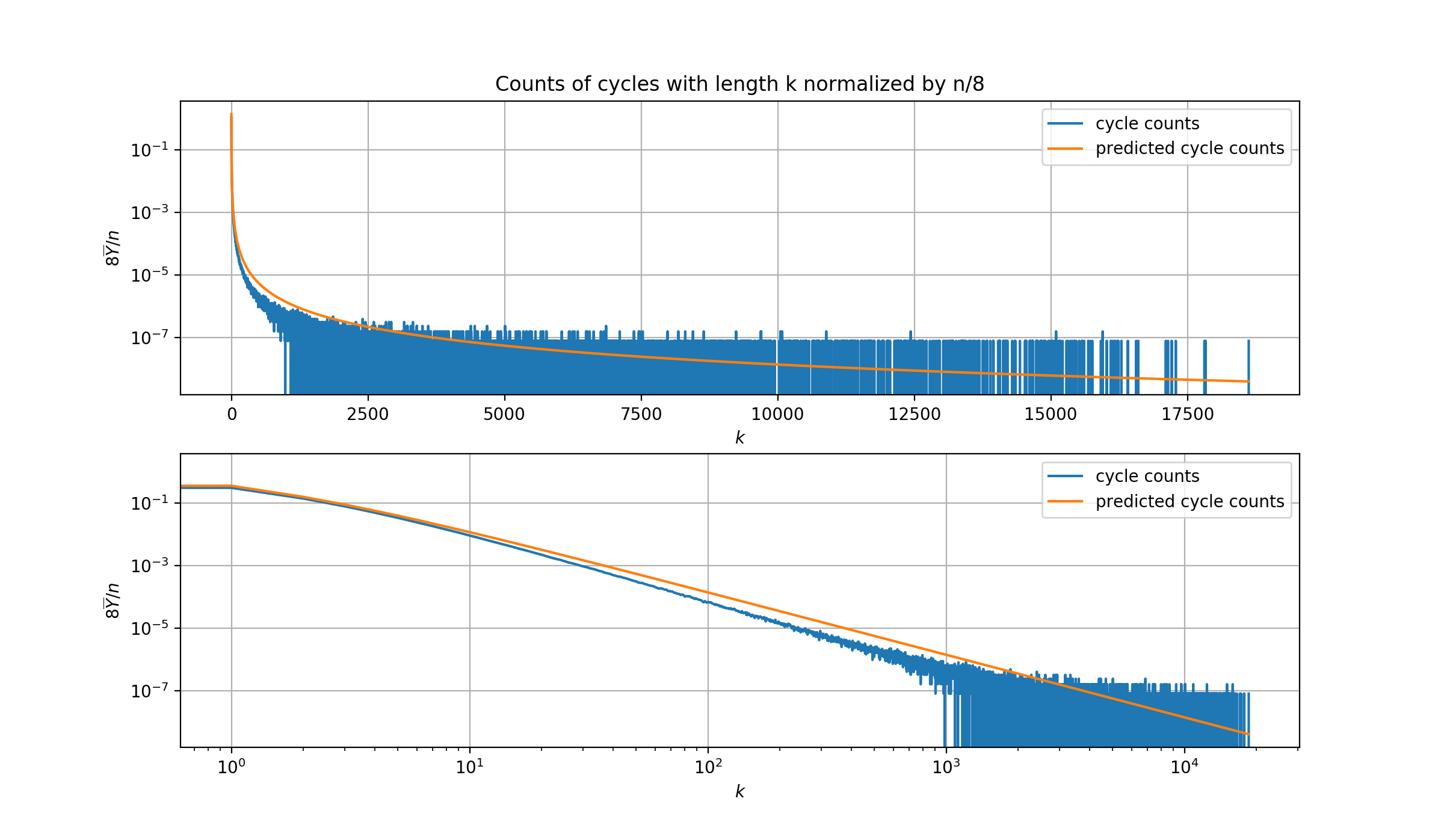}
              \caption{This plot shows $8\overline{Y_k}/n$, where $\overline Y_k$ is the average number of cycles of half-length $k$ in a random meander of half-length $n = 10^5$. The averaging is done over a sample of $1000$ random meanders. For the prediction curve, $8\ovln{Y_k}/n$ is set equal to $k^{-2}$. The upper and lower plots have linear and logarithmic scales along the $x$-axis, respectively. }
              \label{figDynamicSpectrumA}
\end{figure}


\begin{propo}
Let $c_{k}^{(cl)}(M_n)$ be the number of cluster cycles of half-length $k$ in a random m.s. $M_n$ on $2n$ points. 
Then,  for every $k > 0$, 
\bal{
\E c_{k}^{(cl)}(M_n) \lesssim 2 n \gamma^{-k},
}
where $\gamma = 1.2402\ldots$.
\end{propo}
\begin{proof}
We have 
\bal{
c_{k}^{(cl)}(M_n) = \sum_\pi c_{k,\pi}^{(cl)}(M_n).
}
The number of shapes of half-length $k$ is equal $R_k$, and for every fixed $k \geq 1$ we use asymptotic (\ref{asymptotic}) in order to find that 
\bal{
\E c_{k}^{(cl)}(M_n) \sim  R_k 2n 16^{-k}.
}
By using the exact upper bound on the number of meanders $R_k$ (\ref{upperboundRk}), we have
that for every fixed $k\geq 1$, 
\bal{
\E c_{k}^{(cl)}(M_n) \lesssim 2 n \Big(\frac{16}{12.901}\Big)^{-k} = 2 n (1.2402...)^{-k}.
}
\end{proof}

From numerics, however, it appears that for a wide range of $k$, the function $\E c_{k}(M_n)$ (the expected number of cycles of half-length $k$) does not exhibit exponential decline in $k$. Instead, it appears that for large $n$ and slowly growing $k_n \ll n$, this function can be more accurately described as $ c n k_n ^ {-\beta}$ where $\beta \approx 2$. (For larger values of $k$ it appears that $\beta$ becomes smaller but the decline is still polynomial, not exponential.) This suggests that for large $k$ non-cluster cycles play a significant role.  See Figure \ref{figDynamicSpectrumA} for an illustration. In this illustration $n = 10^5$ and the largest cycle is expected to be of order $10^4$. 



\subsection{Simply generated random meander systems}

In this section we describe an extension of Theorem \ref{theoNumberRinglets} to a more general class of distributions on meander systems. The results are similar to the case of the uniform distribution and are not used later. While we provide all necessary definitions, the reader unfamiliar with simply generated trees can safely skip this section.

In order to define simply generated m.s., we recall a bijection between NC pairings on $2n$ points and the set of rooted planar trees on $n + 1$ vertices.

 We denote this bijection as $\beta: \TT_{n + 1} \to \PP_n$,  where $\TT_{n + 1}$ is the set of all rooted planar trees on $n+1$ vertices, and $\PP_n$ is the set of non-crossing pairings of $[2n]$. 
 
  \begin{wrapfigure}{r}{.4 \textwidth}
              \includegraphics[width=  \linewidth]{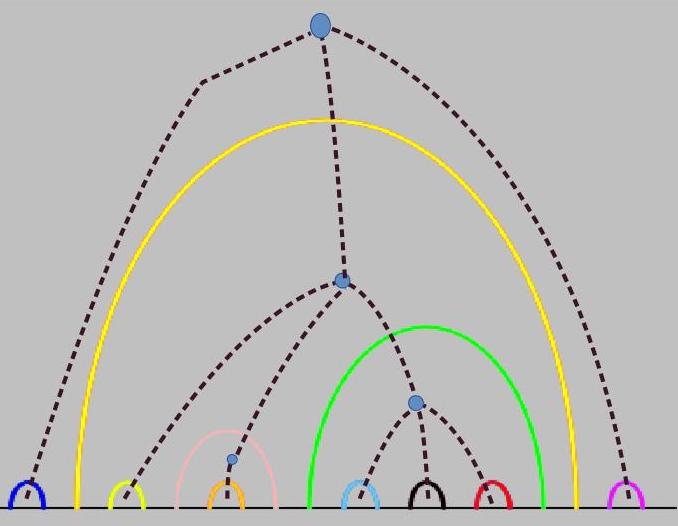}
              \caption{A bijection between NC pairings and planar trees.}
              \label{figBijectionPairingTree}
\end{wrapfigure}

 The bijection is illustrated in Figure \ref{figBijectionPairingTree}.\footnote{Formally, a tree is mapped to a Dyck walk path and then the path is mapped to an NC pairing. Both maps are standard, see, for example, note I.48 on p.77 in \cite{flajolet_sedgewick2009}.} It preserves the internal structure of pairings and trees in the sense that  a pair in a pairing corresponds to an edge of the corresponding tree.  (The edge crosses the arc of the corresponding pair.) If we identify the tree edge $(u, v)$ from parent $u$ to its child $v$ with vertex $v$, then we obtain a bijection between pairs of a pairing and non-root vertices of the corresponding planar tree.  In particular,  every pair $(x, x+1)$ corresponds to a tree leaf. More generally, a pair $(x, x + 2k + 1)$ corresponds to a vertex of out-degree  $k$.

Now, recall the definition of  simply-generated random rooted planar trees $\TT_n$ on $n$ vertices. Let $w_0 > 0, w_1, w_2, \ldots $ be a sequence of non-negative numbers (``weights''). Define the weight of a tree $T$ as 
\bal{
w(T) = \prod_{d = 0}^\infty w_d^{N_d(T)},
}
where $N_d(T)$ denote the number of vertices of out-degree $d$ in the tree $T$.

A \emph{simply-generated random tree} with weight sequence $w = (w_0, w_1, w_2, \ldots )$ is a tree in $\TT_n$ chosen at random with probability $w(T)/Z_n$, where $Z_n = \sum_{T\in \TT_n} w(T)$.

The uniform distribution on planar random trees corresponds to the weight sequence $w_i = 1$ for all $i$ (and to equivalent weight sequences). Simply generated trees can be alternatively defined in terms of Galton-Watson trees conditioned to have a fixed number of vertices. See \cite{janson2012} for an overview of these topics.

\begin{defi}
Let $\vec w = (w_0>0, w_1, w_2, \ldots)$ and $\vec{w}' = (w'_0>0, w'_1, w'_2, \ldots)$ be two sequences of non-negative weights, and let $(T_n, T'_n)$ be two independent simply generated trees on $n + 1$ vertices with weight sequences $\vec w$ and $\vec w'$, respectively. Then, a \emph{simply generated meander system} on 2n points is the pair $(P_n, P'_n)$, where $P_n$ and $P'_n$ are non-crossing partitions of $[2n]$ obtained by the bijection $\beta$ from the trees $T_n$ and $T'_n$, respectively. 
\end{defi}

We call the simply generated meander system \emph{symmetric} if the weight sequences for the upper and lower pairings are the same,  $\vec w' = \vec{w}$.

For  example, a random meander system with the uniform distribution is a symmetric simply generated m.s. with weights  $\vec  w' = \vec w = (1, 1, 1, \ldots)$. 

 For a weight sequence $(w_0, w_1, \ldots)$, let $\theta(t) := \sum_{k=0}^\infty w_k t^k$, and let $\rho$ be the radius of convergence for this series. Define   
$\psi(t) := \frac{t \theta'(t)}{\theta(t)}$. 
One can check that $\psi(t)$ is a non-decreasing function on $[0,\rho)$. Define 
\bal{
\tau =
\begin{cases}
\inf \{t: \psi(t) = 1\}, &\text{ if }\sup_{t \uparrow \rho} \psi(t) \geq 1,
\\
\rho, & \text{ otherwise.}
\end{cases}
}

\begin{theo}
\label{theoRinglets}
Let $M_n$ be a symmetric simply generated meander system on $2n$ points, and let $c(M_n)$ denote the number of connected components of $M_n$. Then, 
\bal{
\E \, c(M_n) \gtrsim \frac{1}{2}\Big(\frac{w_0}{\theta(\tau)}\Big)^2 n.
}
\end{theo}

\textbf{Remark:} For non-symmetric meander systems the result of Theorem \ref{theoRinglets} is also likely to hold, with a suitable modification on the coefficient before $n$. However, in order to use our method, one needs the asymptotic equidistribution property of the pairs $(x, x + 1)$ in interval $[1, 2n]$, for either lower or upper NC pairing, and this property is not obvious in the general, non-uniform case.  


\begin{proof}[Proof of Theorem \ref{theoRinglets}]
Let $c_o(M_n)$ denote the number of ringlets in the meander system $M_n$. Obviously, $c(M_n) \geq c_o(M_n)$ and it is enough to give a lower bound for $\E\, c_o(M)$. 

As in the proof of Theorem \ref{theoNumberRinglets},
\bal{
\E c_o(M_n) &=  \sum_{x = 1}^{2n - 1} p_n(x)^2, 
}
where $p_n(x)$ is the probability that $(x, x+1)$ is in a given random NC pairing of $[2n]$.


The sum  $\sum_{x = 1}^{2n - 1} p_n(x)$ equals the expected number of pairs of the form $(x, x + 1)$ in a random pairing. The pairs of this form are mapped to leaves of the corresponding tree by the bijection $\beta^{-1}$ . Hence, $\sum_{x = 1}^{2n - 1} p_n(x) = \E N_0(T_n) $, where $N_0(T_n)$ denotes the number of leaves in a tree $T_n$.   

By Theorem 7.11 in \cite{janson2012}, we have that $N_0(T_n)/n$ converges in probability to $\pi_0 = w_0/\theta(\tau)$, which by dominated convergence implies that $\E\big[N_0(T_n)/n\big] \to \pi_0$, that is, $\E N_0(T_n) = \sum_{x = 1}^{2n - 1} p_n(x) \sim \pi_0 n$.

Then, by the Cauchy-Schwarz inequality, 
\bal{
(2n - 1)\sum_{x = 1}^{2n - 1} p_n(x)^2 \geq \Big(\sum_{x = 1}^{2n - 1} p_n(x)\Big)^2 = (\E N_0(T_n))^2 \sim \Big(\frac{w_0 n}{\theta(\tau)}\Big)^2. 
}
This implies that
\bal{
\E c_o(M_n) \gtrsim \frac{1}{2}\Big(\frac{w_0}{\theta(\tau)}\Big)^2 n. 
}
\end{proof}

%

\subsection{Random rainbow m.s.}

%

There is another method to introduce a probability measure on a \emph{subset} of  meander systems, which have recently appeared in the literature. 
A \emph{rainbow pairing} of the set $[2m]$ is the pairing $P_1 (m) = \{(1, 2m), (2, 2m - 1), \ldots, (m, m + 1)\}$. Its geometric representation consists of $m$ arcs that look like a rainbow. A shifted rainbow pairing is a pairing of the set $\{k, k +1, \ldots, 2m + k - 1\}$, for some $k \in \mathbb{N}$, which is defined as $P_k (m) = \{(k, 2m + k - 1), (k + 1, 2m +k - 2), \ldots, (m + k - 1, m + k)\}$.

\begin{wrapfigure}{l}{.4 \textwidth}
              \includegraphics[width= \linewidth]{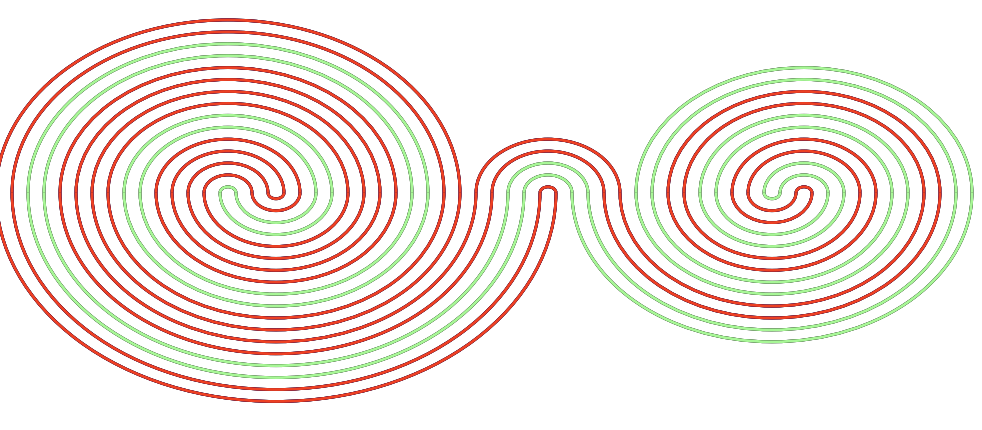}
              \caption{A rainbow meander system with parameters  $\alpha = [15, 5, 11]$ and
    $\beta = [18, 13]$.}
              \label{figRainbowMeander}
\end{wrapfigure}

A \emph{rainbow meander system} \\ $M_n(\alpha_1, \ldots, \alpha_s | \beta_1, \ldots, \beta_t)$ has the upper and lower pairings which are  unions of shifted rainbow pairings,  $P_u = \cup_{i=1}^s P_{r_i}(\alpha_i)$, where $r_i = 1 + 2\sum_{j = 1}^{i - 1} \alpha_j$, and $P_l = \cup_{i=1}^t P_{r'_i}(\beta_i)$, where $r'_i = 1 + 2\sum_{j = 1}^{i - 1} \beta_j$. It is required that $\sum_{i = 1}^s \alpha_i = \sum_{i = 1}^t \beta_i = n$. See illustration in Figure  \ref{figRainbowMeander}.

A \emph{random rainbow meander system} of type $(s, t)$ on  the set $[2n]$  is a rainbow m.s. $M_n(\vec \alpha| \vec\beta)$, where integer vectors $\vec\alpha = (\alpha_1, \ldots, \alpha_s)$ and $\vec\beta = (\beta_1, \ldots, \beta_t)$ are taken uniformly and independently in the simplices  $\sum_{i = 1}^s \alpha = n$, $\alpha_i \geq 1$ and $\sum_{i = 1}^t \beta_i = n$, $\beta_i \geq 1$, respectively. 

The constraint $\sum_{i = 1}^s \alpha_i = n$ implies that random variables $\alpha_i$ are not independent. A simpler random system relaxes this constraint by allowing a random number of points $2n$.  Namely, a \emph{relaxed rainbow meander system} of type $(s, 1)$ and size $N$ is a rainbow m.s. $M_N(\vec \alpha | \beta)$  where coordinates of the vector $\vec \alpha = (\alpha_1, \ldots, \alpha_s)$ are taken independently and uniformly in the set $[N]$, and $\beta = \sum_{i = 1}^s \alpha_s$. 

Explicit formulas for the number of cycles in a rainbow meander system are known only for meander systems of type $(2, 1)$ and $(3, 1)$, that is, only for meander systems that have one rainbow at the bottom and two or three rainbows at the top. In those cases, they are
\begin{equation}
\label{gcd2}
c(M) = \gcd(\alpha_1, \alpha_2),
\end{equation}
 and 
 \begin{equation}
 \label{gcd3}
c(M) = \gcd(\alpha_1 + \alpha_2, \alpha_2 + \alpha_3),
\end{equation}
 respectively. (See \cite{fiedler_castaneda2012}.)
 
 \begin{theo}
 \label{theoRainbowMeanders}
 Let $M_{2,N}$ be a relaxed random rainbow meander of type $(2,1)$ and size $N$. Then for the number of cycles $c(M_{2,N})$, the following equations hold:
 \bal{
 \P\Big(c(M_{2,N}) = x\Big) &= \frac{6}{\pi^2} \frac{1}{x^2} + O\Big(\frac{\log (N/x)}{Nx}\Big),
 \\
 \E\Big(c(M_{2,N})\Big) &= \frac{6}{\pi^2} \log N + C + O\Big(\frac{\log (N)}{\sqrt{N}}\Big),
 \\
 \E\Big(c(M_{2,N})^2\Big) &= \frac{N}{3} \Big(\frac{2 \zeta(2)}{\zeta(3)}  - 1\Big)+ O(\log N),
 }
 where $\zeta(z)$ is Riemann's zeta function. 
 \end{theo}
 
  \begin{proof}
  For $c(M_{2,N})$, the formula (\ref{gcd2}) is applicable, in which $\alpha_1$ and $\alpha_2$ are taken uniformly and independently at random from interval $[N]$. Then, all claims of the theorem follow from the classical results for the greatest common divisor of two random integers. See, for example, \cite{diaconis_erdos2004}. 
 \end{proof}
 
 This result shows that the number of cycles in a random rainbow m.s. of type $(2, 1)$ behaves very differently than in a general random m.s. First of all, the probability that a rainbow m.s. has some fixed small number of cycles converges to a positive number as the size of m.s. grow. For example, the probability that the m.s. has only one cycle converges to $\frac{6}{\pi^2}$. This is dramatically different from the situation for general random m.s., where this probability is exponentially small in the size of the m.s. 
 
 Second, the expected number of cycles is logarithmic in $N$, while in a general random m.s. this quantity is linear in the size of the system $n$. This suggests that a typical rainbow m.s. has a small number of long cycles, in contrast to a typical general m.s. where it has a large number of small cycles.
 
 Finally, the variance of the number of cycles is of order $N$, which implies that with some small probability a rainbow m.s. of type $(2,1)$ still can have a large number of small cycles. 
 
 For random rainbow m.s. of higher types, the question about  the distribution of the number of cycles is more difficult. Even for the type $(3,1)$, for which an explicit formula (\ref{gcd3}) is available, the calculation of the distribution of the $\gcd(\alpha_1 + \alpha_2,\alpha_2 + \alpha_3)$  is difficult. The numerical evaluations show that the law in this case is the same as for the type $(2,1)$:  $\P\Big(c(M_{3,N}) = x\Big) \sim \frac{6}{\pi^2} \frac{1}{x^2}$ for large $N$.
 
 In a recent paper, \cite{dgzz2019}, it was shown that if a m.s. $M_N$ is taken uniformly at random among all meander systems with the length of at most $2N$ and with the total of $p\geq 4$ leaves in the trees corresponding to the upper and lower pairings of $M_N$, then 
 \bal{
 \P\Big(c(M_N) = 1\Big) \sim \frac{1}{2}\Big(\frac{2}{\pi^2}\Big)^{p-3} \binom{2p - 4}{p - 2},
 }
 when $N \to  \infty$. For $p = 4$, this expression agrees with the numerical evaluation for the rainbow m.s. of type $(3,1)$, mentioned in the previous paragraph. Unfortunately, the method in \cite{dgzz2019} does not generalize immediately to  probabilities  $\P\Big(c(M_N) = x\Big)$ for $x > 1$.

%

\section{The largest cycle}
\label{sectionLargestCycle}

\begin{figure}[htbp]
\centering
              \includegraphics[width= 0.9\textwidth]{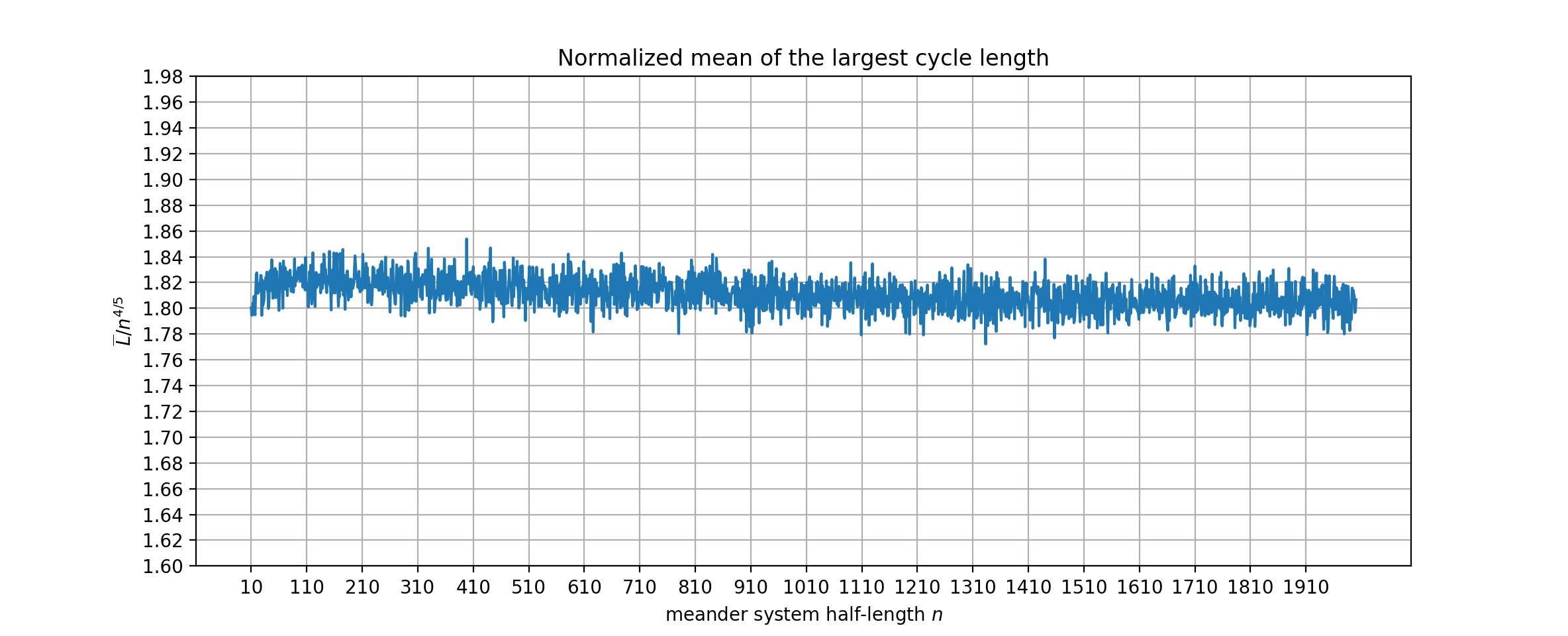}
              \caption{The length of the largest cycle normalized by dividing it by $n^{4/5}$, where $n$ is the half-length of the meander system. The horizontal axis shows the half-size of the meander $n$, $10 \leq n \leq 2000$. The plotted half-size of the largest cycle is the average over a sample of $4000$ independent uniformly distributed meander systems.}
              \label{figSizeLargestCycle}
\end{figure}

For random rainbow m.s. of size $n$ and type $(2, 1)$, and more generally of type $(s, t)$, the results in Theorem \ref{theoRainbowMeanders} and in \cite{dgzz2019} show that the probability that a meander system has a single cycle converges to a positive number as $n \to \infty$. In particular, this implies that the expected size of the largest cycle is $\Omega(n)$. For random meander systems with uniform distribution the situation is different. 

Figure \ref{figSizeLargestCycle} shows results of numerical simulations for random meander systems of size $n$. They suggest that the largest cycle in a random m.s. of half-size $n$ has length around $n^{4/5}$. 

This behavior is surprising. Usually, when we consider a random system in which a concept of a ``cycle'' is well-defined, we encounter a dichotomy between the expected number of cycles and the length of the largest cycle.  For example, for a random permutation in the permutation group $S_n$, the expected number of cycles is $\sim \log n$ and the expected length of the largest cycle is proportional to $n$. For a random partition of the set $[n] = \{1, \ldots, n\}$, the expected number of blocks is $\sim n/\log n$ while the expected length of the largest block is proportional to $\log n$. 

In contrast, for random meander systems, Theorem \ref{theoNumberRinglets} establishes that the number of cycles grows at a rate proportional to $n$. However, numeric simulations suggest that  the expected  largest cycle length also grows at a polynomial rate, as $n^\alpha$, where $\alpha$ is close to $4/5$. 

\begin{figure}[htbp]
\centering
              \includegraphics[width= 0.9 \textwidth]{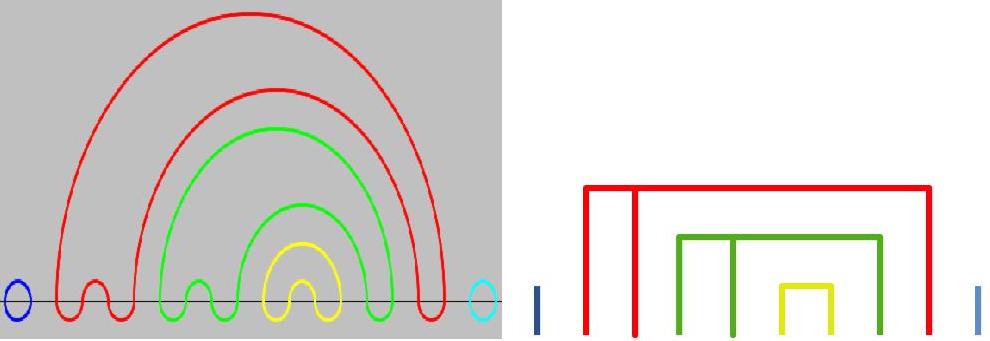}
              \caption{An illustration of the bijection between comb-like meander systems and non-crossing partitions}
              \label{figMeanderNCPartition}
\end{figure}

This behavior  becomes even more surprising if we compare it with that of another type of random meander systems, \emph{comb-like meander systems}. These systems have  the lower non-crossing pairing $[(1, 2) (3, 4), \ldots, (2n - 1, 2n)]$. It is not difficult to see that these meander systems can be bijectively mapped to non-crossing (``NC'') partitions, in such a way that cycles of a meander system correspond to blocks of the image partition. This bijection is illustrated in Figure \ref{figMeanderNCPartition}.

Due to this bijection, the comb-like meander systems is much easier to analyze. In particular, from the result about non-crossing partitions, it follows that the expected number of cycles is $(n + 1)/2$, its variance $\sim n/8$, and that the number of cycles is asymptotically normal for large $n$. 

Moreover, by using another bijection, from NC partitions to rooted planar trees, it is not difficult to prove that the distribution of cycle lengths is geometric, meaning that for large $n$, the expected number of cycles of half-length $l$ is  $\sim n/2^{l + 1}$. One can also calculate the covariances between numbers of cycles with half-length $l$ and $l'$, and prove some limit distribution theorems for these numbers. (See \cite{kortchemski_marzouk2017} for corresponding results for non-crossing partitions.)

For the largest block in a non-crossing partition, we have the following result. 

\begin{theo}
\label{theoSizeLargestBlock}
Let $L_n$ denote the size of the largest block in a uniformly distributed random NC partition of $[n]$. 
Then, as $n \to \infty$, 
\bal{
\frac{L_n}{\log_2 n} \pcv 1,
}
where the convergence is in probability. 
\end{theo}
Note that $\log_2 n \approx 1.443 \log n$, and therefore the largest block in an NC partition is on average shorter than in a usual set partition where it is around $e \log n \approx 2.718 \log n$.

\begin{coro}
Let $M_n$ be a uniformly distributed random comb-like meander on $2n$ points. Then, the length of the largest cycle in $M_n$ is $\sim 2\log_2 n$ in probability.
\end{coro}

\begin{figure}[htbp]
\centering
              \includegraphics[width= 0.7  \textwidth]{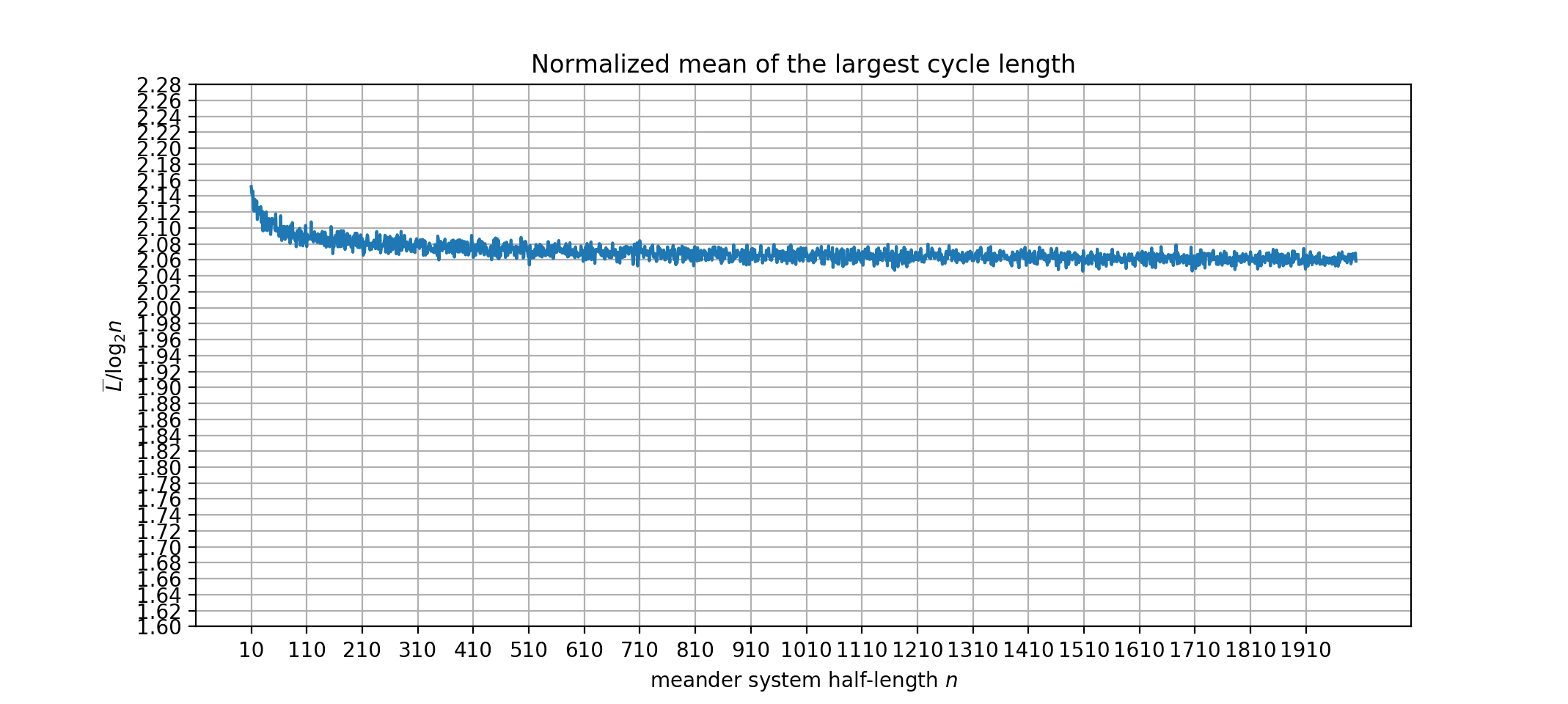}
              \caption{The largest cycle length in a random comb-like m.s. normalized by dividing it by $\log_2 n$, where $n$ is the half-length of the meander system. The horizontal axis shows the half-size of the random m.s. $n$, $10 \leq n \leq 2000$. The plotted  largest cycle length is the average over a sample of $4000$ independent comb-like m.s.}
              \label{figSizeLCinCombMeanders}
\end{figure}

This result is illustrated by a plot in Figure \ref{figSizeLCinCombMeanders}. Note especially that the normalization in this plot is different from the normalization in Figure \ref{figSizeLargestCycle}. Here the length of the cycle is divided by $\log_2 n$, while in Figure \ref{figSizeLargestCycle} it was divided by $n^{4/5}$.

\begin{figure}[htbp]
\centering
              \includegraphics[width= 0.4 \textwidth]{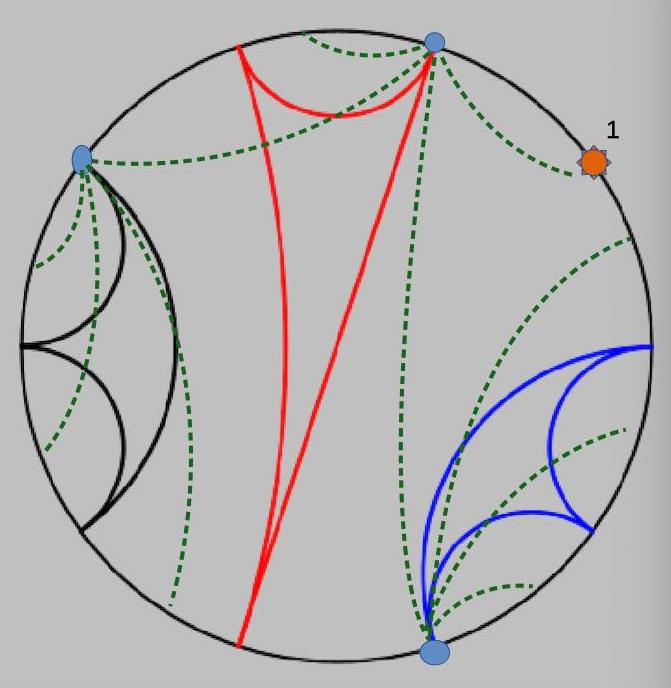}
              \caption{A bijection between NC partitions and rooted planar trees. This example shows a partition of $[n] = [10]$ with $b= 4$ blocks mapped to a tree on $n + 1 = 11$ vertices with $b = 4$ internal vertices, and $7$ leaves.  }
              \label{figBijectionPartitionTree}
\end{figure}

\begin{proof}[Proof of Theorem \ref{theoSizeLargestBlock}]

It is easy to map bijectively NC partitions to rooted ordered trees. This can be done either as in \cite{kortchemski_marzouk2017} or as indicated in Figure \ref{figBijectionPartitionTree}. These bijections have the property that a block of size $k$ corresponds to a vertex of out-degree $k$. 

Then we can use a result about the distribution of the maximum out-degree in a uniformly distributed random tree. 

In particular, Theorem 19.3 for simply generated trees in Janson \cite{janson2012} implies that 
\bal{
\frac{L_n}{\log n} \to \frac{1}{\log(\rho/\tau)} = \frac{1}{\log 2}, 
}
since for uniformly distributed trees, the weights are $w_k = 1$, the radius of the convergence of the weight-generating function $\rho = 1$, and the parameter $\tau = 1/2$. 
\end{proof}

By Example 19.17 in \cite{janson2012}, the distribution of the largest degree in a uniform random tree can be approximated by the distribution of a maximum of i.i.d geometric random variables with parameter $p = 1/2$. This can be translated in terms of the largest cycle length in a comb-like meander. However, even when $n \to \infty$, there is no limiting distribution and the behavior of the maximum depends on how the binary expansion of $n$ looks like. (There is a significant difference between such cases as $n = 2^k$ and $n = 2^k - 1$.) Still, it is possible to formulate the following result.

(We use notation $\lfloor t \rfloor$  for the largest integer $\leq t$, and $\{t\} = t - \lfloor t \rfloor$ for the fractional part of $t$. )

\begin{theo}
\label{theoDistrLargestBlock}
Let $L_n$ denote the largest cycle half-length in a uniformly distributed random comb-like meander system on $2n$  points. Let  $\alpha(n) = 2^{\{\log_2 n\}}\in [1, 2)$. Then, as $n \to \infty$, 
\bal{
\P\Big[L_n  - \lfloor \log_2 n \rfloor \leq x\Big] =  \exp\Big( -\alpha(n) 2^{-(x + 1)}\Big) \Big( 1 + O\big(n^{-1}\log^2 n\big)\Big).
}
\end{theo}

So, apart from the factor of $\alpha(n)$, the limiting distribution is the double exponential distribution, that often arises in the study of extreme value distributions. 

This result can be proved by using the relation to geometric random variables (see Example 19.17 in \cite{janson2012}). We give a direct proof using generating functions in Appendix. 

Now let us return from comb-like meanders to the case, in which both the upper and lower NC pairings are random. As seen in Figure \ref{figSizeLargestCycle},  numerical simulations suggest that the growth in the largest cycle length is approximately $n^{\alpha}$ with $\alpha \approx 4/5$. The theorem below shows rigorously that the growth is at least logarithmic. 

%
 
 \begin{theo}
 \label{theoLargestCycle}
 Let $L_n$ be the largest cycle half-length in a uniformly distributed random meander system of half-length $n$. Then, for all sufficiently large $n$, 
 \bal{
 \E L_n \geq 0.8384 \log_2 n.
 }
 \end{theo}
 

 For the proof we need a couple of lemmas. 
 
 \begin{lemma}
 Suppose $k_n$ is an integer valued non-negative function of $n$ and $1 \ll k_n \ll n$. Then 
 \begin{align}
 \label{equExp}
 \Big(\frac{C_{n - k_n}}{C_n}\Big)^2 &= 2^{-4k_n} \bigg(1  + 3\frac{k_n}{n} + O
 \Big[\big(\frac{k_n}{n}\big)^2\Big]\bigg)
 \\
 \label{equCov}
 \Big(\frac{C_{n - 2 k_n}}{C_n}\Big)^2 - \Big(\frac{C_{n - k_n}}{C_n}\Big)^4
  &= 2^{-8k_n}\bigg(3 \Big(\frac{k_n}{n}\Big)^2 
  +  O\Big[\big(\frac{k_n}{n}\big)^3\Big] \bigg)
 \end{align}
 \end{lemma}
 \begin{proof}
 We use the asymptotic expansion for the Catalan numbers:
 \bal{
 C_n = \frac{4^n}{\sqrt{\pi n^3}}\bigg(1 - \frac{9}{8}n^{-1} + \frac{145}{128} n^{-2} + 
 \frac{1155}{1024}n^{-3} + O\big(n^{-4}\big)\bigg),
 }
 see Figure VI.3 on p. 384 in \cite{flajolet_sedgewick2009}. 
 
 Then, we have
 \bal{
 \frac{C_{n - k_n}}{C_n} = 2^{-2k_n} \bigg[ \frac{1}{(1 - k_n/n)^{3/2}} 
 \frac{1 - \frac{9}{8}\frac{1}{1 - k_n/n} n^{-1} + \frac{145}{128} \frac{1}{(1 - k_n/n)^2} n^{-2} +O(n^{-3}) }{1 - \frac{9}{8}n^{-1} + \frac{145}{128} n^{-2} +O(n^{-3}) }\bigg].
 }
 By expanding this expression in powers of $k_n/n$ and $n^{-1}$, we obtain the asymptotic series for this ratio. Then, formulas (\ref{equExp}) and (\ref{equCov}) can be obtained by manipulating this asymptotic series. 
 \end{proof}

 Let $I_x$ be the indicator of the event that a random meander system contains a cluster cycle $C_x$ with support $(x + 1, x + 2, \ldots, x + 2k)$.  (Recall that by Definition \ref{defiClusterCycle} cluster cycles have no gaps in their support.) Then we have 
 \begin{equation}
 \label{equExpIx}
 \E I_x =  M_k \Big(\frac{C_{n - k}}{C_n}\Big)^2 \1_{0 \leq x \leq 2n - 2k}. 
\end{equation} 
 where $M_k$ is the number of proper meanders of half-length $k$. 
 Similarly, if $\max (x, y) \leq 2n - 2k$, then 
 \begin{equation}
 \label{equVarIx}
 \E (I_x I_y) =
 \begin{cases}
 \E I_x, & \text{ if } x = y, 
 \\
 0, & \text{ if } x \ne y \text{ and } |x - y| < 2k,
 \\
 M_k^2  \Big(\frac{C_{n - 2k}}{C_n}\Big)^2,& \text{ if }  |x - y| \geq 2k.
 \end{cases}
 \end{equation}
 If $\max (x, y) > 2n - 2k$, then $\E (I_x I_y) = 0$.
 
 Let $\Cl_n(k)$ denote the number of cluster cycles of half-length $k$ in a random meander system of half-length $n$. We aim to prove the following lemma.  
 \begin{lemma}
 \label{lemmaPoisson}
 Let $1\ll k_n \ll n$. Then, for $n\to \infty$, 
 \begin{equation}
 \E \Cl_n(k_n) \sim \Var \big(\Cl_n(k_n)\big) \sim 2n M_{k_n} 2^{-4{k_n}},
 \end{equation}
 \end{lemma}
 \begin{proof}
 By formula (\ref{equExpIx}), we have 
 \bal{
  \E \Cl_n(k_n) &= \sum_{x = 0}^{2n - 2k_n} \E I_x = \sum_{x = 0}^{2n - 2k}  M_{k_n} \Big(\frac{C_{n - k_n}}{C_n}\Big)^2 \sim 2n M_{k_n}  2^{-4k_n},
 }
 where in the last step we used (\ref{equExp}).
 For the variance we have:
  \bal{
  \Var \Big(\Cl_n(k_n)\Big) = \E\Big( \Cl_n(k_n)^2\Big) -  \Big(\E \Cl_n(k_n)\Big)^2 
  = \sum_{x, y = 0}^{2n - 2k_n}   \Big[\E (I_x I_y) - \E I_x \E I_y\Big].
  }
  The double sum over $x$ and $y$ can be split in 3 cases, according to whether (i) $x = y$, or (ii) $x \ne y$ and $|x - y| < 2k_n$, or (iii) $|x - y| \geq 2k_n$. In the first case, the number of terms is $\sim 2n$, in case (ii) it is $ \sim 4k_n\times 2n$ and in the last case it is $\sim 4n^2$. By using formulas in (\ref{equExp}), (\ref{equCov}), (\ref{equExpIx}), (\ref{equVarIx}), we write: 
    \bal{
  \Var \Big(\Cl_n(k_n)\Big) &\sim 2n \Big[ M_{k_n} 2^{-4k_n} - \big(M_{k_n} 2^{-4k_n}\big)^2\Big] 
  - 8k_n n \Big[ \big(M_{k_n} 2^{-4k_n}\big)^2 \Big]
  \\
  &+ 4n^2 \Big[3 \big(\frac{k_n}{n}\big)^2 \big(M_{k_n} 2^{-4k_n}\big)^2\Big]
  \\
  &\sim  2n  M_{k_n} 2^{-4k_n},
  }
  where we used the rigorous bound that $M_k \leq 13^k$ for sufficiently large $k$, which implies that 
 $k_n M_{k_n} 2^{-4k_n} \ll 1$. In addition, for the last term we used that $k_n \ll n$. 
 \end{proof}

 \begin{proof}[Proof of Theorem \ref{theoLargestCycle}]
 In order to prove that $\E L_n \geq c \log_2 n$, it is enough to show that for all sufficiently large $n$, with probability  $\geq 1/2$, there is a cycle of length $\geq 2c \log_2 n$. We will show a stronger statement that with probability $\geq 1/2$ (and all sufficiently large $n$), there is a \emph{cluster cycle} of length $\geq 2c \log_2 n$. 
 
We use a well-known inequality (see Theorem 4.3.1 in \cite{alon_spencer2000}) that for a non-negative integer-valued random variable $X$, it is true that 
\begin{equation}
\label{equAlonSpencer}
 \P(X = 0) \leq \frac{\Var(X)}{(\E X)^2}.
 \end{equation}
 We apply this inequality to $X = \Cl_n(k_n)$, with $k_n = c \log_2 n$. Then, by Lemma  \ref{lemmaPoisson}, 
 \bal{
 \frac{\Var(X)}{(\E X)^2} = \Big[2n M_{k_n} 2^{-4 k_n}\Big]^{-1}= \frac{2^{4 k_n}}{2 n M_{k_n}} \leq n^{4c - 1 - 2.8073 c} = n^{1.1927 c - 1}, 
 }
 where we used a conservative estimate $M_{k_n} \geq 7^{k_n} = 2^{2.8073 k_n}$. Hence, for $c \leq 0.8384$, this ratio $\to 0$ as $n \to \infty$. This implies that for $k_n \leq 0.8384 \log_2 n$, the probability $\P\big[\Cl_n(k_n) \geq 1\big] \to 1$, which completes the proof of the theorem.
 \end{proof}

 \section{Remarks.} 
 \label{sectionRemarks}
 
  1. The behavior of the largest cycle length appears to be universal with respect to a large set of models.  Plots in Figure \ref{figLargestCycleBinaryM} illustrate that the expected largest cycle length is proportional to $n^{4/5}$ not only for random m.s. with the uniform distribution, but also for other random m.s., such as simply-generated m.s. and semimeanders, in which the upper pairing is chosen uniformly at random and the lower pairing is fixed to be the rainbow $(1, 2n)$, $(2, 2n - 1)$, \ldots, $(n, n+1)$. Similar results are also observed for systems not reported in Figure \ref{figLargestCycleBinaryM}, in particular, for m.s. with weight sequence $w_0 = w_k = 1$, $w_i = 0$ for $i \ne k$, with $k = 2, 3, 4$; and for m.s. for which the weight sequence declines polynomially, $w_k = (k + 1)^{-\beta}$, for $\beta = 2$.  
 
  \begin{figure}[htbp]
\centering
              \includegraphics[width= 0.8 \textwidth]{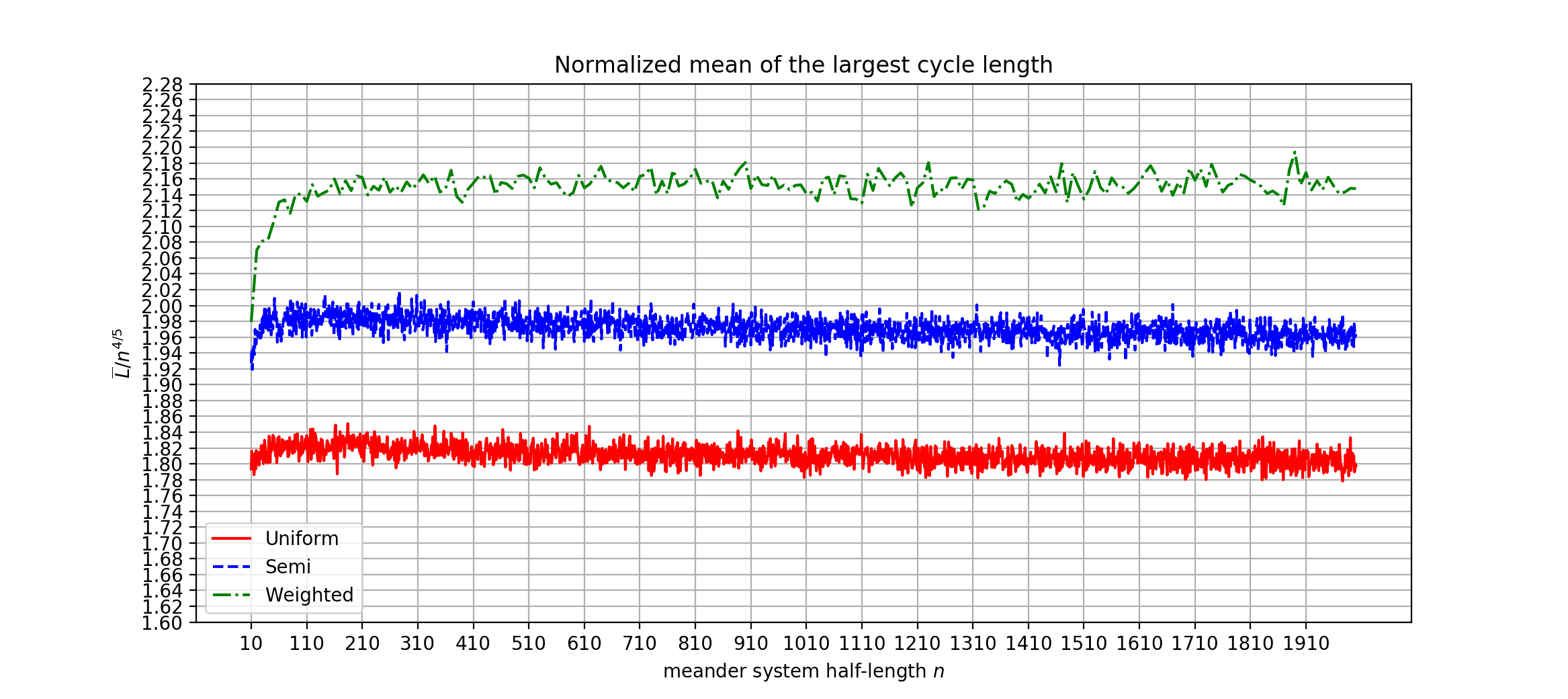}
              \caption{Plots of the largest cycle length normalized by $n^{4/5}$, where $n$ is the half-size of the m.s. In this plot $10 \leq n  \leq 2000$ and each point is the average of the largest cycle length over a sample of $4000$ independent m.s. A random ``semimeander'' systems  have its upper pairing chosen uniformly at random while its lower pairing is the ``rainbow'': $[(1, 2n), (2, 2n - 1), \ldots, (n, n + 1)]$. The weighted meander system in this example has weights $(1, 1, 1, 0, \ldots)$. [It is more time consuming to generate these m.s. so to save time they were generated for half-sizes $\{10, 20, \ldots, 2000\}$ only.]} 
              \label{figLargestCycleBinaryM}
\end{figure}

\begin{figure}[htbp]
\centering
              \includegraphics[width= 0.8 \textwidth]{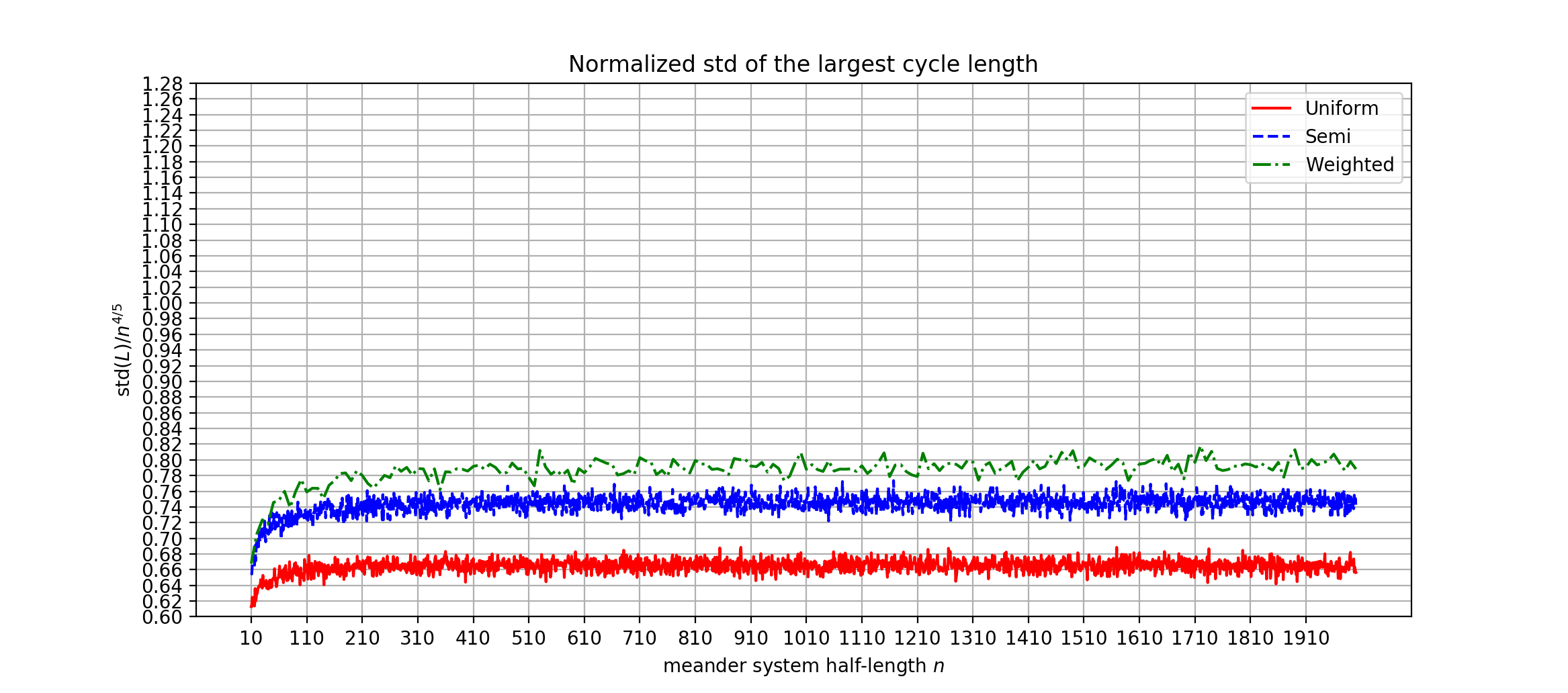}
              \caption{The estimate of the standard deviation of the largest cycle normalized by $n^{4/5}$, where $n$ is the half-size of the meander system. In the plot  $10 \leq n  \leq 2000$. The standard deviation shown in the graph is the average over a sample of  $4000$ independent random meander systems.}
              \label{figLargestCycleVariance}
\end{figure}
 
 2. The size of the largest cycle does not concentrate with the growth in $n$, in the sense that the standard deviation of this random variable seems to grow at the same rate ($n^{4/5}$) as the expectation. See an illustration in Figure \ref{figLargestCycleVariance}.

 \begin{figure}[htbp]
\centering
              \includegraphics[width= 0.7 \textwidth]{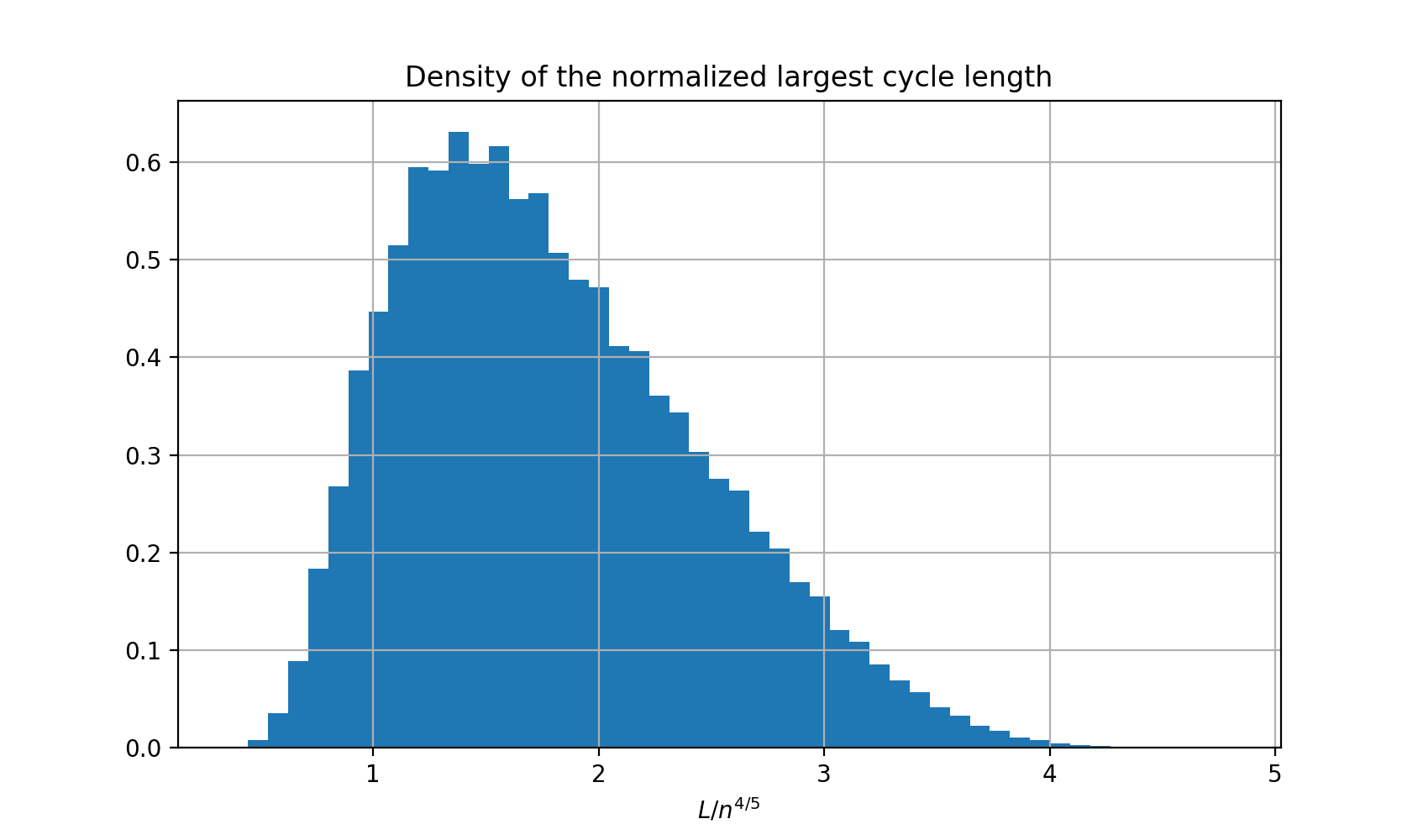}
              \caption{The histogram of the largest cycle length normalized by $n^{4/5}$. The meander system half-length $n = 2000$ and the size of the sample used for estimation is $10^6$.}
              \label{figHistLargestCLength}
\end{figure}
 
3. The distribution of the largest cycle is shown in Figure \ref{figHistLargestCLength}. Its skewness $\approx 0.59$ and it appears as one of the usual extreme value distributions.

 
  \begin{figure}[htbp]
\centering
              \includegraphics[width= 0.7 \textwidth]{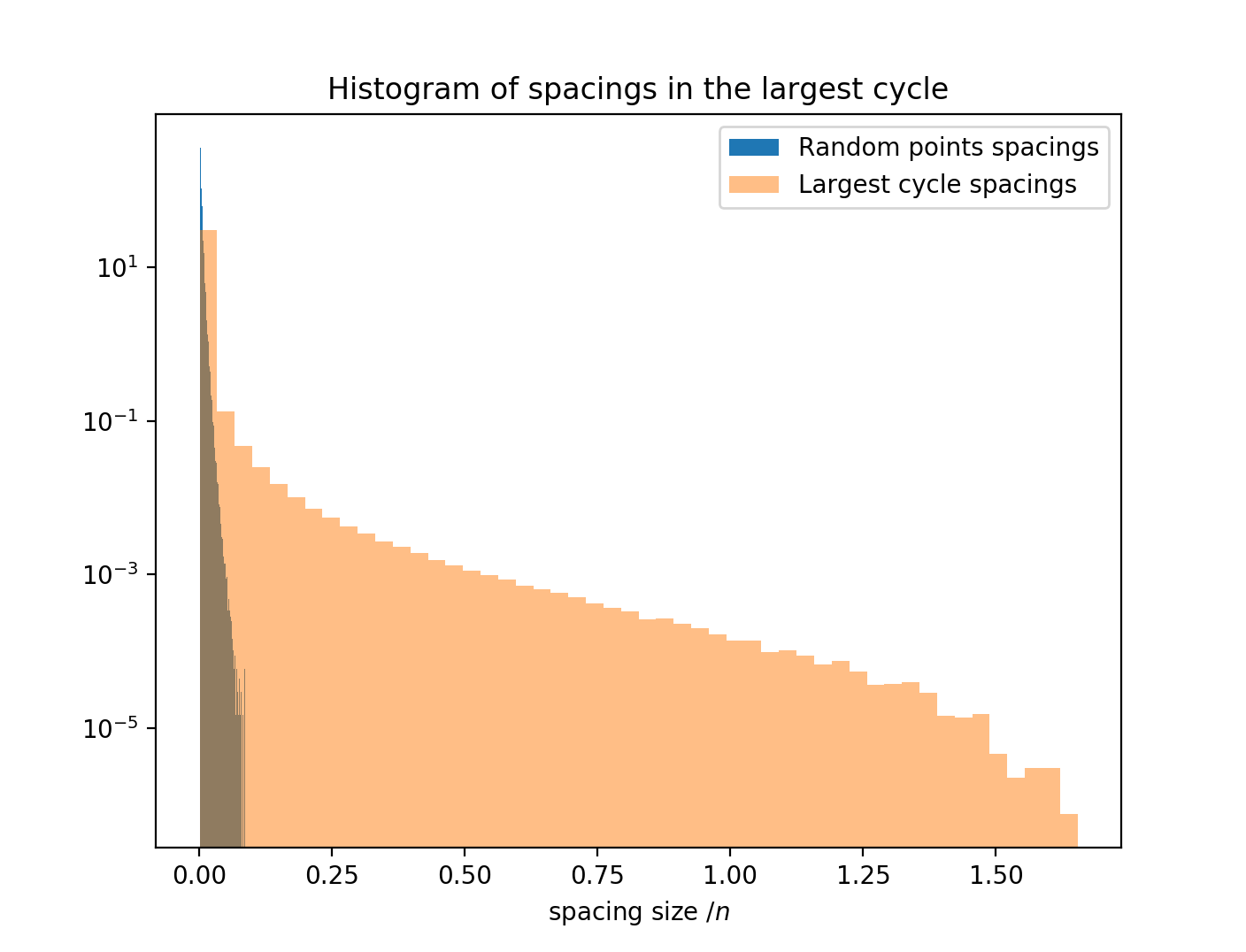}
              \caption{The histogram of the largest cycle spacings [normalized by $n$]. Each m.s. has half-length $n = 2000$ and the sample size of m.s. used to build the sample of largest cycle spacings is $5\times 10^4$. For comparison, a histogram of spacings for independent random points is shown. It shows spacing distribution for a matched sample of points which are equal in number to the largest cycle length but are chosen  uniformly at random on interval $[2n]$.}
              \label{figHistLargestCSpacings}
\end{figure}
 
 4. The spacings in the support of the largest cycle also exhibit an interesting  structure. 
 Spacings of size 1 (which occur when x and x + 1 are both  in the support of the cycle) take around $80\%$ of all spacings. Spacings of size 3 are about $10\%$ of all spacings. (Spacings of size 2 are impossible.) These percentages seem to be insensitive to the size of the meander system. The distribution of spacings is shown in Figure \ref{figHistLargestCSpacings}. (Note the logarithmic scale of $y$-axis.) The distribution of spacings in the largest cycle  differs significantly from the distribution of spacings in the case of uniformly distributed points. In particular, there are significantly more large spacings than one would expect in a sample of independent random points. 
 
 3. The connected components of meander systems appears in other areas of mathematics. For example, in \cite{dergachev_kirillov2000}, the number of connected components is used to evaluate the index of \emph{seaweed Lie algebras}. In \cite{dgg1997}, this quantity arises in the connection with the trace of Temperley-Lieb matrices. 
 
 4. The common meaning of the word ''meander'' is a winding curve followed by a river. So, it is perhaps relevant to mention that some scaling laws have been experimentally discovered for the behavior of river flows. For example, in 1957,  Hack studied the drainage system of rivers in the Shenandoah valley and the adjacent mountains of Virginia, and formulated the law, which is now called Hack's law. Namely, there is a power law relation $l \sim a^{\beta}$, with $\beta$ close to $0.6$, between the length $l$ of a longest stream in a connected system of streams and the area of the basin that collects the precipitation contributing to the stream through its tributaries. See a review and some scaling arguments in \cite{giacometti1996} 
 and a recent mathematical model in \cite{roy_saha_sarkar2016}. 
 
 Hack's law, however, reflects the behavior of random trees rather than that of random meander systems. For random trees, related results are well-known. For example, it is known that the height of a simply-generated random tree on $n$ vertices is proportional to $n^{1/2}$ (under mild conditions on the weight sequence). (See, for example, \cite{flajolet_sedgewick2009} or \cite{janson2012}.) In contrast, in the case of meander systems, it is not clear which scaling mechanisms contribute to the exponent $4/5$ observed in numerical experiments.

 %

\appendix
\section{Appendix}
\label{sectionAppendix}
\begin{proof}[Proof of Proposition \ref{propoStaples}]  As in the proof of Theorem \ref{theoNumberRinglets},  we can use indicator functions to derive the formula
\begin{equation}
\label{numberStaples}
\E c_{st}(M) = 2 \sum_{x, y \colon 1 \leq x + 1 < y \leq 2n - 2}   p_1(x, y) p_2(x, y), 
\end{equation}
where $p_1(x, y)$ is the probability that $(x, x+1)$ and $(y, y+1)$ are pairs in a random pairing, and $p_2(x, y)$ is the probability that  $(x, y+1)$ and $(x + 1, y)$ are pairs. The coefficient $2$ is because we count both upper and lower staples. 

For $p_1(x,y)$ we have 
\begin{equation}
\label{p1}
p_1(x, y) = \frac{C_{n - 2}}{C_n} = \frac{1}{16} + O(n^{-1}),
\end{equation}
where $C_i$ are Catalan numbers and the constant in the O-term is uniform over $x$ and $y$.

For $p_2(x, y)$, we look at staples of different length $k$ separately. (Here the length of a staple with support $(x, x+1, y, y+1)$ is defined as $k = (y - x)/2$.) For a staple of length $k$, we have 
\begin{equation}
\label{p2}
p_2(x, y) = \frac{C_{n - k - 1} C_{k - 1}}{C_n},
\end{equation}
which depends only on $k$ and $n$.
Fix an $\eps \in (5/6, 1)$ and consider the sums of probabilities $p_2(x, y)$ over all $x$ and $y$ that satisfy additional conditions on length $k$. 
\bal{
S_1 &= \sum p_2(x, y), \text { where $k \leq n^\eps$},
\\ 
S_2 &= \sum p_2(x, y), \text { where $ n^\eps < k \leq n - n^\eps$},
\\
S_3 &= \sum p_2(x, y), \text { where $ n - n^\eps < k \leq n$}.
}
For $S_2$ we use the asymptotic approximation to Catalan numbers, and make the following estimate for each term in the sum, 
\bal{
p_2(x, y) &\sim \frac{4^{n - k - 1}}{\sqrt{\pi}(n - k - 1)^{3/2} } \frac{4^{k - 1}}{\sqrt{\pi}(k - 1)^{3/2} }
/\frac{4^{n}}{\sqrt{\pi}n^{3/2} }
\\
& = O\Big(n^{3/2(1 - 2\eps)}\Big) = o(n^{-1}).
}
Since the number of terms in the sum $S_2$ is no greater than $O(n^2)$, we conclude that $S_2 = o(n)$.

For $S_3$, we can write the sum as follows,
\bal{
S_3 &=  \sum_{k = n - n^\eps}^{n - 1} \sum_{x = 1}^{2n - 2k - 1} p_2(x, x + 2k)
\\
&= O\Big( n^\eps \sum_{k = n - n^\eps}^{n - 1}  C_{n - k - 1}   \frac{C_{k - 1}}{C_n}\Big),
}
where we used formula (\ref{p2}) and the observations that $p_2(x, x + 2k)$ does not depend on $x$ and that the number of terms in the inner sum is $\leq 2 n^\eps$. We can continue by changing the summation index, 
\bal{
S_3 &= O\Big( n^\eps \sum_{l = 0}^{n^\eps}  C_{l}   \frac{C_{n - l - 2}}{C_n}\Big)
\\
&= O\Big( n^\eps \sum_{l = 0}^{\infty}  C_{l} 4^{-l} \Big) = O(n^\eps),
}
where the last line follows because the sum $\sum_{l = 0} ^\infty C_l 4^{-l}$ is convergent. 

Finally, we address the sum $S_1$,
\bal{
S_1 &=  \sum_{k = 1}^ {n^\eps} \sum_{x = 0}^{2n - 2k} p(x, x + 2k)
\\
&\sim  2n \sum_{k = 1}^ {n^\eps} C_{k - 1}   \frac{C_{n - k - 1}}{C_n}
 \sim 2n \sum_{k = 1}^ {n^\eps} C_{k - 1} 4^{-(k + 1)}
\\
& \sim \frac{n}{8} \sum_{l = 0}^\infty C_{l} 4^{-l}.
}
The last sum can be calculated by noting that $C_l = \frac{1}{2\pi} \int_0^4 x^{l - \frac{1}{2}} \sqrt{4 - x} \, dx$, which leads to $\sum_{l = 0}^\infty C_{l} 4^{-l} = 2$ and so $S_1 \sim n/4$.

From the estimates on $S_1$, $S_2$, $S_3$, and formulas (\ref{numberStaples}) and (\ref{p1}), it follows that 
\bal{
\E c_{st}(M) \sim 2 \times \frac{1}{16} \frac{n}{4} = \frac{n}{32}.
}
\end{proof}

\begin{proof}[Proof of Proposition \ref{propoRings}]
We aim at estimating the sum
\bal{
\E c_O(M) = \sum_{x < y} p(x,y)^2,
}
where $p(x,y)$ is the probability that $(x, y)$ is a pair in a random pairing. 
If $y - x$ is even this probability is zero, otherwise, 
\bal{
 p(x, x + 2k + 1)  = \Big[ \frac{C_k C_{n - k - 1}}{C_n}\Big]^2.
}
So we have 
\bal{
\E c_O(M) &= \sum_{x = 1}^{2n - 1} \sum_{k = 0}^{\lfloor (2n - x - 1)/2\rfloor}  \Big[ \frac{C_k C_{n - k - 1}}{C_n}\Big]^2
\\
&=
\sum_{k=0}^{n - 1} \sum_{x = 0}^{2n - 2k - 1} 
\Big[ \frac{C_k C_{n - k - 1}}{C_n}\Big]^2
\\
&= \sum_{k=0}^{n - 1} (2n - 2k)\Big[ \frac{C_k C_{n - k - 1}}{C_n}\Big]^2
}
 Very similar to the previous proof, by using the asymptotic approximations for the Catalan numbers, we can show that the main contribution is given by the rings with radius $k < n^\eps$, where $\eps < 1$ is appropriately large.   Then, we can apply asymptotic approximation to $C_{n - k - 1}$ and write:
 \bal{
\E c_O(M) &\sim  \sum_{k = 0}^{n^\eps} (2n - 2k) C_k^2 4^{-2 k - 2}
\\
&\sim \frac{n}{8} \sum_{k = 0}^{\infty} \big[C_k 4^{-k }\big]^2 = \Big(\frac{2}{\pi} - \frac{1}{2}\Big) n,
}
where the proof of  the last equality can be found in \cite{lando_zvonkin92}, see their Proposition in Section 5.10.  
\end{proof}

\begin{proof}[Proof of Theorem \ref{theoDistrLargestBlock}]
Because of the bijection between comb-like meanders and non-crossing partitions, it is enough to prove the corresponding result about the largest block in a non-crossing partition of $[n]$. First, we find the generating function for NC partitions with blocks whose length is $\leq k$. The symbolic formula for the class of these partitions is 
\bal{
C^{(k)} = \eps + SET_1(Z) \times SEQ_1(C^{(k)}) + \ldots + SET_k(Z) \times SEQ_k(C^{(k)}).
}
 Here $\eps$ denotes the empty partition, $Z$ is an atom (that is, an element of a partition), and $\mathrm{SET}_{k}(Z) \times \mathrm{SEQ }_{k}(C)$ corresponds to a block of size $k$ containing a marked element (``root''), together with a sequence of NC partitions which are nested between the  elements of this block. See Figure \ref{figNCpartition} for illustration. 

\begin{figure}[htbp]
\centering
              \includegraphics[width=0.4\textwidth]{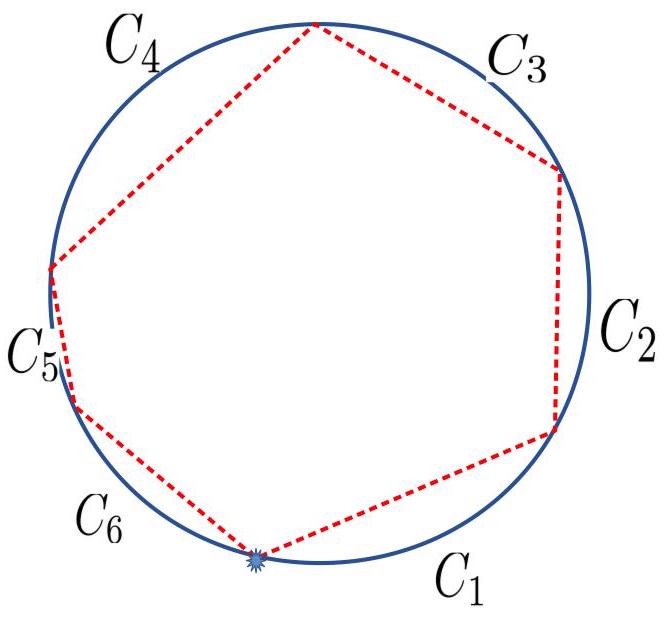}
              \caption{Construction of a non-crossing partition with a block of size 6}
              \label{figNCpartition}
\end{figure}

This leads to the following equation for the generating function:
\bal{
C^{(k)}(z) &= 1 + z C^{(k)}(z) + \ldots + \Big[zC^{(k)}(z)\Big]^k
\\
&=\frac{1 - \Big[zC^{(k)}(z)\Big]^{k + 1}}{1 - zC^{(k)}(z)},
}
or 
\bal{
\Big[zC^{(k)}(z)\Big]^{k + 1} - z C^{(k)}(z)^2 + C^{(k)}(z) - 1 = 0.
}

The coefficient before $z^n$ in the power expansion of $C^{(k)}(z)$, denoted $[z^n]C^{(k)}(z)$, equals the number of NC partitions of $[n]$, in which all blocks have lengths $\leq k$. The total number of NC partitions of $[n]$ is the Catalan number $C_n$. So, the probability that a random NC partition has no block larger than $k$ is $[z^n]C^{(k)}(z)/C_n$, and we are left with the task of evaluating the asymptotic of $[z^n]C^{(k)}(z)$.

Let us use the notation 
$
y(z) = zC^{(k)}(z).
$
Then the equation for $y(z)$ is $y = G(z, y)$, where $G(z, y) = z + y^2 - z y^{k + 1}$. 

Now, let us summarize some tools from the book by Flajolet and Sedgewick, which allows us to write the asymptotic expressions for the coefficients of $y(z)$. 

 We say that that a function of complex argument $y(z)$ is an \emph{analytic generating function (analytic GF)} if it is analytic at zero and if  its expansion, 
 \begin{equation}
 \label{seriesY}
 y(z) = \sum_{n = 0}^{\infty} y_n z^n, 
 \end{equation}
 have real non-negative coefficients $y_0 = 0$, $y_n \geq 0$.
 \begin{defi}
 \label{defiStableSingularity}
 An analytic GF $y(z)$ is said to \emph{have a stable dominant singularity}\footnote{Flajolet and Sedgewick say that $y(z)$ belongs to the \emph{smooth implicit-function schema}.} at $z = r > 0$, if there exists a bivariate function $G(z, w)$ such that 
 \begin{equation}
 y(z) = G(z, y(z)), 
 \end{equation}
 and $G(z, y(z))$ satisfies the following conditions:
\begin{enumerate}[(A)]
 \item $G(z, w) = \sum_{m, n \geq 0} g_{m, n} z^m w^n$ is analytic in a domain $|z| < R$ and $|w| < S$ for some $R, S > 0$.
 \item Coefficients $g_{m, n}$ are non-negative reals, $g_{0, 0} = 0$, $g_{0, 1}\ne 1$ and $g_{m, n} > 0$ for some $m$ and for some $n \geq 2$. 
 \item  The singularity $r < R$ and there exists $s$ such that $0 < s < S$ and that
 \begin{align}
 \label{equ1}
 G(r, s) &= s,
 \\
 \label{equ2}
 G_w(r, s) &= 1,
 \end{align} 
\end{enumerate}
 \end{defi}
 The function $G(z, w)$ is called the \emph{characteristic function} for $y(z)$ and the system (\ref{equ1}), 
(\ref{equ2}) -- the \emph{characteristic system}. 
 
 The condition in (C) is aimed to ensure that $r$ is a singularity of $y(z)$ with $y(r) = s$. The conditions in (A) and (B), especially the non-negativity of the coefficients, ensure that this singularity is a quadratic singularity with the smallest absolute value among all singularities of $y(z)$ (which is why we call it ``the stable dominant singularity'').  
 %
 The following result has  a long history, - a weaker version can be found already in the classic book by Hille \cite{hille62}, -- Theorem 9.4.6 on p. 274 of volume I. In the form similar to the following, it was proved by Meir and Moon in \cite{meir_moon89}. We give the formulation from Flajolet-Sedgewick book \cite{flajolet_sedgewick2009}, cf. Theorem VII.3 on p. 468. 
 
 \begin{theo}[Meir - Moon]
 \label{theoStableSingularity}
 Let $y(z)$ be an analytic GF that has a stable singularity at $r$ with the characteristic function $G(z, w)$. Then the series in (\ref{seriesY}) converges at $z = r$ and 
 \begin{equation}
 y(z) = s - \gamma \sqrt{1 - z/r} + O(1 - z/r),
 \end{equation}
 in a neighborhood of $z = r$, where $s = y(r)$ and 
 \begin{equation}
 \gamma = \sqrt{\frac{2 r G_z(r, s)}{G_{ww}(r, s)}}.
 \end{equation}
 \end{theo}

We apply this result to $y(z) = zC^{(k)}(z)$. 
The singularity point $\big(z_0, y_0 = y(z_0)\big)$  solves the characteristic system:
\bal{
y &= z + y^2 - z y^{k + 1},
\\
1 &= 2y - (k + 1)z y^k.
}
The first equation of the system gives 
\bal{
z = \frac{y(1 - y)}{1 - y^{k + 1}},
}
and after plugging this expression into the second equation we obtain:
\bal{
y = \frac{1}{2}\Big[1 + k y^{k + 1} - (k - 1) y^{k + 2}\Big].
}
The solution for this equation is 
\bal{
y_0 = \frac{1}{2} + \frac{k + 1}{2^{k + 3}} + O\Big(\frac{k}{2^{2k}}\Big).
}
And then, 
\bal{
z_0 = \frac{y_0(1 - y_0)}{1 - y_0^{k + 1}} = \frac{1}{4}\Big(1 + \frac{1}{2^{k + 1}}\Big) + O\Big(\frac{k^2}{2^{2k}}\Big),
}
In order to apply Theorem \ref{theoStableSingularity} we also compute
\bal{
\gamma = \sqrt{\frac{2 z_0 P_z(z_0, y_0)}{P_{yy}(z_0, y_0}} = \frac{1}{2} + O\Big(\frac{k^2}{2^{k}}\Big)
}
and conclude that 
\bal{
y(z) = y_0 - \gamma \sqrt{1 - z/z_0} + O(1 - z/z_0).
}
Using the power expansion for the square root and Theorem VI.4 in \cite{flajolet_sedgewick2009} to justify that the error term in the formula for $y(z)$ can be neglected, we find that 
\bal{
[z^n] y(z) = \gamma \Big[\frac{1}{2\sqrt{\pi}n^{3/2}} + O(n^{-5/2})\Big]\Big(\frac{1}{z_0}\Big)^n.
}
For Catalan numbers (total number of NC partitions of $[n]$) we have the asymptotic approximation
\bal{
C_n = 4^n\Big[\frac{1}{\sqrt{\pi n^{3/2}}} + O\Big(\frac{1}{n^{5/2}}\Big)\Big].
}
Hence we have the following estimate,
\bal{
\P\{L_n \leq k\} &= \frac{[z^n]C^{(k)}(z)}{C_n} = \frac{[z^{n + 1}]y(z)}{C_n} 
\\
&= 4^{-(n + 1)}z_0^{-(n+1)}\Big(1 + O(n^{-1}) + O(k^2/2^k)\Big).
\\
&= \Big[\Big(1 + \frac{1}{2^{k + 1}}\Big) + O\Big(\frac{k^2}{2^{2k}}\Big)\Big]^{-(n+1)}\Big(1 + O(n^{-1}) + O(k^2/2^k)\Big).
}
If we use $k := \lfloor \log_2 n \rfloor + x$ and recall that we defined $\alpha(n) = 2^{\{\log_2 n\}}$,  then we find that  $2^{k + 1} = 2^{\lfloor \log_2 n \rfloor + x + 1} =2^{x + 1} \frac{n}{\alpha(n)}$. Plugging this into the previous expression, we find that 
\bal{
\P\{L_n \leq k\} &=  \Big[1 + \alpha(n) 2^{-(x + 1)} \frac{1}{n} + O\Big(\frac{\log^2 n}{n^2}\Big)\Big]^{-(n+1)}\Big(1 + O\big(\frac{\log^2 n}{n}\big)\Big)
\\
&=\exp\Big( - \alpha(n) 2^{-(x + 1)}\Big) \Big( 1 + O\big(n^{-1}\log^2 n\big)\Big).
}
\end{proof}

\bibliographystyle{plain}
\bibliography{comtest}

\begin{thebibliography}{10}

\bibitem{albert_paterson2005}
M.~H. Albert and M.~S. Paterson.
\newblock Bounds for the growth rate of meander numbers.
\newblock {\em Journal of Combinatorial Theory, Series A}, 112:250--262, 2005.

\bibitem{alon_spencer2000}
Noga Alon and Joel~H. Spencer.
\newblock {\em The Probabilistic Method}.
\newblock John Wiley and Sons, Inc., second edition, 2000.

\bibitem{arnold88}
V.~I. Arnold.
\newblock The branched covering \mbox{$CP^2 \to S^4$}, hyperbolicity and
  projective topology.
\newblock {\em Sibirsk. Mat. Zh.}, 29(5):36 --47, 1988.

\bibitem{dgzz2019}
Vincent Delecroix, Elise Goujard, Peter Zograf, and Anton Zorich.
\newblock Enumeration of meanders and \mbox{M}asur-\mbox{V}eech volumes.
\newblock arXiv: 1705.05190, March 2019.

\bibitem{dergachev_kirillov2000}
Vladimir Dergachev and Alexandre Kirillov.
\newblock Index of \mbox{L}ie algebras of seaweed type.
\newblock {\em Journal of Lie Theory}, 10(331-343), 2000.

\bibitem{diaconis_erdos2004}
Persi Diaconis and Paul Erd\"{o}s.
\newblock On the distribution of the greatest common divisor.
\newblock {\em Institute of Mathematical Statistics Lecture Notes}, 45:56--61,
  2004.

\bibitem{fiedler_castaneda2012}
Bernold Fiedler and Pablo Castaneda.
\newblock Rainbow meanders and \mbox{C}artesian billiards.
\newblock {\em Sao Paulo Journal of Mathematical Sciences}, 6(2):247 -- 275,
  2012.

\bibitem{flajolet_sedgewick2009}
Philippe Flajolet and Robert Sedgewick.
\newblock {\em Analytic Combinatorics}.
\newblock Cambridge University Press, 2009.

\bibitem{dgg1997}
P.~Di Francesco, O.~Golinelli, and E.~Guitter.
\newblock Meanders and the \mbox{T}emperley-\mbox{L}ieb algebra.
\newblock {\em Communications in Mathematical Physics}, 186(1):1 -- 59, 1997.

\bibitem{dgg2000}
P.~Di Francesco, O.~Golinelli, and E.~Guitter.
\newblock Meanders: exact asymptotics.
\newblock {\em Nuclear Physics B}, 570:699--712, 2000.

\bibitem{franz_earnshaw2002}
Reinhard O.~W. Franz and Berton~A. Earnshaw.
\newblock A constructive enumeration of meanders.
\newblock {\em Annals of Combinatorics}, 6:7 -- 17, 2002.

\bibitem{fukuda_nechita2019}
Motohisa Fukuda and Ion Nechita.
\newblock Enumerating meandric systems with large number of loops.
\newblock {\em Annales de l'Institut Henri Poincare D Combinatorics, Physics
  and their Interactions}, 6(4):607--640, 2019.

\bibitem{hille62}
Einar Hille.
\newblock {\em Analytic Function Theory}, volume~1.
\newblock Ginn and Company, 1962.

\bibitem{janson2012}
Svante Janson.
\newblock Simply generated trees, conditioned \mbox{G}alton--\mbox{W}atson
  trees, random allocations and condensation.
\newblock {\em Probability Surveys}, 9:103--252, 2012.

\bibitem{jensen2000}
Iwan Jensen.
\newblock A transfer matrix approach to enumeration of plane meanders.
\newblock {\em Journal of Physics A: Mathematical and General}, 33:5953 --
  5963, 2000.

\bibitem{jensen_guttmann2000}
Iwan Jensen and Anthony~J. Guttmann.
\newblock Critical exponents of plane meanders.
\newblock {\em Journal of Physics A: Mathematical and General}, 33:L187--L192,
  2000.

\bibitem{kortchemski_marzouk2017}
Igor Kortchemski and Cyril Marzouk.
\newblock Simply-generated non-crossing partition.
\newblock {\em Combinatorics, Probability and Computing}, 26(4):560--592, 2017.

\bibitem{lando_zvonkin92}
S.~K. Lando and A.~K. Zvonkin.
\newblock Meanders.
\newblock {\em Selecta Mathematica Sovietica}, 11(2):117 -- 144, 1992.

\bibitem{lando_zvonkin93}
S.~K. Lando and A.~K. Zvonkin.
\newblock Plane and projective meanders.
\newblock {\em Tneortical Computer Science}, 117:227--241, 1993.

\bibitem{meir_moon89}
A.~Meir and J.~W. Moon.
\newblock On an asymptotic method in enumeration.
\newblock {\em Journal of Combinatorial Theory, Series A}, 51:77-- 89, 1989.

\bibitem{giacometti1996}
Riccardo Rigon, Ignacio Rodriguez-Iturbe, Amos Maritan, Achille Giacometti,
  David~G. Tarboton, and Andrea Rinaldo.
\newblock On \mbox{H}ack's law.
\newblock {\em Water Resources Research}, 32(11):3367 -- 3374, November 1996.

\bibitem{roy_saha_sarkar2016}
Rahul Roy, Kumarjit Saha, and Anish Sarkar.
\newblock Hack's law in a drainage network model: \mbox{A} brownian web
  approach.
\newblock {\em Annals of Applied Probability}, 26(3):1807 -- 1836, 2016.

\end{thebibliography}
\end{document}